\documentclass[11pt,reqno]{amsart}
\usepackage[american]{babel}
\usepackage{amsmath}
\usepackage{dsfont,mathtools,amssymb}
\usepackage{tikz}
\usepackage{graphicx}
\usepackage{subcaption}
\usepackage{color}
\usepackage{mathrsfs}
\usepackage{hyperref}
\mathtoolsset{showonlyrefs,showmanualtags}

\definecolor{darkblue}{rgb}{0.1,0.1,0.7}

\definecolor{darkred}{rgb}{0.7,0.1,0.1}

\newtheorem{theorem}{Theorem}[section]
\newtheorem{proposition}[theorem]{Proposition}
\newtheorem{lemma}[theorem]{Lemma}
\newtheorem{corollary}[theorem]{Corollary}
\newtheorem{remark}[theorem]{Remark}

\newtheorem{conjecture}[theorem]{Conjecture}

\numberwithin{equation}{section}
\numberwithin{figure}{section}

\newcommand{\sumtwo}[2]{\sum_{\substack{#1 \\ #2}}} % sum with 2 lines 

\newcommand{\ind}{\mathbf{1}}

\renewcommand{\phi}{\varphi}

\newcommand{\R}{\mathbb{R}}

\newcommand{\IND}{{\bf 1}}
\newcommand{\tmix}{T_{\rm mix}}

\DeclareMathOperator{\cov}{Cov}

\DeclareMathOperator{\Ent}{Ent}

\newcommand{\grad}{\nabla}

\newcommand{\be}{\begin{equation}}

\newcommand{\cB}{\ensuremath{\mathcal B}}

\newcommand{\cE}{\ensuremath{\mathcal E}} 
 
\newcommand{\cG}{\ensuremath{\mathcal G}}

\newcommand{\cL}{\ensuremath{\mathcal L}}

\newcommand{\cQ}{\ensuremath{\mathcal Q}} 
 
\newcommand{\cS}{\ensuremath{\mathcal S}}

\newcommand{\bbN}{{\ensuremath{\mathbb N}} }

\newcommand{\bbR}{{\ensuremath{\mathbb R}} }

\newcommand{\bbZ}{{\ensuremath{\mathbb Z}} }

\newcommand{\si}{\sigma} 
\newcommand{\ent}{{\rm Ent} } 
\newcommand{\smallent}{{\rm ent} } 
\newcommand{\smallvar}{{\rm var} } 
\newcommand{\smallcov}{{\rm cov} } 
\newcommand{\var}{{\rm Var} }

\let\a=\alpha \let\b=\beta   \let\d=\delta  
      \let\k=\kappa  \let\l=\lambda
            
\let\r=\rho      
  
     \let\L=\Lambda 
\let\O=\Omega

\def\({\left(}
\def\){\right)}
\newcommand{\PAR}[1]{{{\left(#1\right)}}} % (1)

\title{Entropy inequalities for random walks and permutations}
\author{Alexandre Bristiel}
\address{ENS Lyon, France}
\email{alexandre.bristiel@ens-lyon.fr}
\author{Pietro Caputo}
\address{Department of Mathematics and Physics, Roma Tre University, Largo San Murialdo 1, 00146 Roma, Italy.}
\email{pietro.caputo@uniroma3.it}

\dedicatory{Dedicated to the memory of Dima Ioffe}

\begin{document}

\begin{abstract}
We consider a new functional inequality controlling the rate of relative entropy decay for random walks, the interchange process
and more general block-type dynamics for permutations. The inequality lies between the classical logarithmic Sobolev inequality and the modified  logarithmic Sobolev inequality, roughly interpolating between the two as the size of the blocks grows. Our results suggest that the new inequality may have some advantages with respect to the latter well known inequalities when multi-particle processes are considered. We prove a strong form of tensorization for independent particles interacting through synchronous updates. Moreover, for block dynamics on permutations we compute the optimal constants  in all mean field settings, namely whenever the rate of update of a block depends only on the size of the block. Along the way we establish the independence of the spectral gap on the number of particles for these mean field processes. As an application of our entropy inequalities we prove a new subadditivity estimate for permutations, which implies a sharp upper bound on the permanent of arbitrary matrices with nonnegative entries, thus resolving a well known  conjecture.    

%
%
%
%\noindent R\'ESUM\'E. On consid\`ere une nouvelle in\'egalit\'e fonctionnelle contr\^olant le taux de d\'ecroissance de l'entropie relative de marche al\'eatoire, processus d'\'echange ou plus g\'en\'eralement, de dynamique en bloc pour les permutations. Cette in\'egalit\'e se situe entre l'in\'egalit\'e logarithmique de Sobolev classique et l'in\'egalit\'e logarithmique de Sobolev modifi\'ee, interpolant les deux lorsque la taille des blocs augmente. Notre r\'esultat sugg\`ere que cette nouvelle in\'egalit\'e pourrait avoir des avantages par rapport \`a ces derni\`eres dans le cadre de processus \`a multiples particules. 
%On prouve une forme de tensorisation forte pour des particules ind\'pendantes \`a mise \`a jour synchronis\'ees. De plus, pour toutes les dynamiques en bloc \`a champ moyen, quand les taux d'actualisation d\'ependent uniquement de la taille des blocs, on \'etablie la constante optimale. Au passage, on prouve aussi l'ind\'ependance du trou spectral par rapport au nombre de particules. En application de notre in\'egalit\'e entropique, on prouve une estim\'ee de sous-additivit\'e pour les permutations, impliquant une borne sup\'erieure optimale pour le permanent de matrices non negatives arbitraire, ce qui r\'esout une conjecture bien connue.
 
\end{abstract}
\maketitle
% \subjclass[2000]{Primary 60K35; secondary 82B20; 82C26.}
\thispagestyle{empty}

\section{Introduction and main results}\label{sec:01}
Given a finite, weighted, undirected graph $G$, consider the continuous time Markov chain with infinitesimal generator
\begin{align}\label{L}
\cL f(x)=\sum_{y\in V}c_{xy}[f(y)-f(x)],
\end{align}
where $V$ denotes the vertex set, $c_{xy}\geq 0$ is the weight along the undirected edge $xy$, and $f:V\mapsto\bbR$ is a generic function. 
We refer to this process as the random walk on $G$, or as the single particle process on $G$. A fundamental quantity in the analysis of random walks on graphs is the spectral gap $\l(G)$, defined as the second smallest eigenvalue of the graph Laplacian $-\cL$. The constant $\l(G)$ is also characterized as the largest constant $\l\ge 0$ such that for all $  f:V\mapsto\bbR$,
\begin{align}\label{vargap}
\l\,\var f \leq \frac2{n} \sum_{x,y\in V}c_{xy}\smallvar_{xy} f,
\end{align}
where $n=|V|$ is the number of vertices, $\var f=\mu(f^2)-\mu(f)^2$ denotes the variance of $f$ with respect to the uniform distribution $\mu$ over $V$, and $\smallvar_{xy} f=\frac14(f(x)-f(y))^2$ is the local variance of $f$ at the edge $xy$. In this paper we introduce an entropic analogue of the  inequality \eqref{vargap}. Namely, we call $\k(G)$ the largest constant $\k\ge 0$ such that for all $  f:V\mapsto\bbR_+$,
\begin{align}\label{kent}
\k\,\ent f \leq \frac2{n} \sum_{x,y\in V}c_{xy}\,\smallent_{xy} f,
\end{align}
where $\ent f = \mu\left[f\log(f/\mu(f))\right]$ denotes the relative entropy of $f$ with respect to $\mu$ and 
\begin{align}\label{ents}
%\ent f = \mu\left[f\log(f/\mu(f))\right]\,,\qquad 
\smallent_{xy} f=\frac12\,f(x)\log \frac{f(x)}{\frac12(f(x)+f(y))} + \frac12\,f(y)\log \frac{f(y)}{\frac12(f(x)+f(y))}
\end{align}
is the local entropy of $f$ at the edge $xy$. The quantity $\smallent_{xy} f$ is $\frac12(f(x)+f(y))$ times the relative entropy of the Bernoulli distribution with parameter $p=f(x)/(f(x)+f(y))$ with respect to the Bernoulli distribution with parameter $1/2$, and it is thus a natural measure of local departure from uniformity of $f$. We refer to $\k(G)$ as the {\em entropy constant} of the graph $G$. 

A standard linearization argument shows that $\k(G)\leq \l(G)$ for any $G$.
%, and that $\k(G)<\l(G)$ if and only if there exists a nonconstant $f:V\mapsto\bbR_+$ which saturates the inequality  \eqref{kent}. 
The classical {\em logarithmic Sobolev inequality}, which characterizes the hypercontractivity of the semigroup $e^{t\cL}$, is obtained as in \eqref{kent} by replacing $\smallent_{xy} f$ with $\smallvar_{xy} \sqrt f$, while the {\em modified logarithmic Sobolev inequality},  which characterizes the rate of exponential decay of the relative entropy along the semigroup $e^{t\cL}$, is obtained as in \eqref{kent} by replacing $\smallent_{xy} f$ with the local covariance $\smallcov_{xy} (f,\log f)=\frac14(f(x)-f(y))\log(f(x)/f(y))$; see e.g.\ \cite{DS,BobTet}. We write $\b(G)$ and $\varrho(G)$ for the associated graph constants. 
Since
\begin{align}\label{ineqents}
2\log(2)\smallvar_{xy}\sqrt f\leq \smallent_{xy}  f\leq 2\,\smallvar_{xy} \sqrt f \leq \frac12 
\,
\smallcov_{xy} (f,\log f),
\end{align}
see Lemma \ref{lem:ele} below, for all weighted graphs $G$ the constant $\k(G)$ satisfies 
\begin{align}\label{ineqentsk}
2\log(2)\b(G)\leq \k(G)\leq 2\b(G)\leq \frac12\,\varrho(G)\leq \l(G).
\end{align}
%the inequality \eqref{kent} is equivalent to log-Sobolev up to a factor $1/\log(2)$ and it is stronger than modified log-Sobolev. 
%The second inequality in \eqref{ineqents} follows from Jensen's inequality, while the first follows from the known fact that   
%the variance of $\sqrt f$ is at most the entropy of $f$ for any underlying distribution; see \cite{LO}.
% We set ${\rm ent}_{xx}(f)=0$ for all $x\in V$. $\smallent_{xy}f=$Log-Sob and mLS see \cite{DSC,BT}
%Moreover, since 
%\begin{align}\label{ineqents1}
%\smallent_{xy} f\leq 2\,\smallvar_{xy}\sqrt f,
%\end{align} 
%(see e.g.\  \cite{DS}), the inequality \eqref{kent} is equivalent to log-Sobolev up to a factor 2. 
As we will see, these relations change when considering generalizations of our inequality to hypergraphs. Moreover, things become particularly interesting when considering generalizations to multi-particle processes. 

\subsection{Hypergraphs}
The hypergraph  version is defined as follows. Given a collection of nonnegative weights $\a=\{\a_A,\,A\subset V\}$, we write $\k[\a]$ for 
the largest constant $\k\ge 0$ such that for all $f:V\mapsto\bbR_+$,
\begin{align}\label{kentalpha}
\k\,\ent f \leq \frac{1}{n} \sum_{A\subset V}|A|\a_A\,\smallent_{A} f,
\end{align}
where 
\begin{align}\label{enta}
%\ent f = \mu\left[f\log(f/\mu(f))\right]\,,\qquad 
\smallent_A f=\frac1{|A|}\sum_{x\in A}f(x)\log (f(x)/\bar f _A)\,,\qquad \bar f _A=\frac1{|A|}\sum_{x\in A}f(x),
\end{align}
is the entropy on the block $A\subset V$.
Note that in the case of blocks of size 2, that is when $\a_A=0$ unless $|A|=2$, then \eqref{kentalpha} is equivalent to \eqref{kent} with the choice $c_{xy}=\a_{A}/2$ whenever $A=xy$.  The reason for the special choice of normalization in \eqref{kentalpha} will become apparent  below. 

The collection of weights $\a$ is viewed as a vector indexed by the subsets of $V$. We refer to the case when the weights $\a_A$ depend only on $|A|$ as the {\em mean field} case. If $\ell\in\{2,\dots,n\}$ is fixed then we write $\a^\ell$ for the vector $\a$ defined by $\a_A=\ind_{|A|=\ell}$, so that the general mean field case has the form $\a=\sum_{\ell=2}^n w_\ell\,\a^\ell$ for some nonnegative vector $w=(w_2,\dots,w_n)$. %\{w_\ell,\,\ell=2,\dots,n\}$.  
\begin{theorem}[One particle, mean field case]\label{th:meanfield1}
Suppose $\a=\sum_{\ell=2}^n w_\ell\,\a^\ell$ for some nonnegative vector $w=(w_2,\dots,w_n)$. Then 
\begin{align}\label{ent1meanfield}
\k[\a]=\sum_{\ell=2}^n w_\ell \,\frac{\binom{n-1}{\ell-1}\log \ell}{\log n},
\end{align}
and equality in \eqref{kentalpha} is uniquely achieved at multiples of a Dirac mass.
\end{theorem}
\begin{remark}\label{rem:K_n}
{\em In particular, the complete graph $G=K_n$, which corresponds to $c_{xy}\equiv 1$, or $\a=2\a^2$, satisfies for all $n\geq 2$,
 \begin{align}\label{kKn}
\k(K_n)=\frac{2(n-1)\log 2}{\log n}.
\end{align}
We recall that the (modified) log-Sobolev constants of the complete graph $K_n$ satisfy $\b(K_2)=1$, $\varrho(K_2)=4$ and, for $n>2$,
 \begin{align}\label{kKnls}
\b(K_n)=\frac{n-2}{\log (n-1)},\qquad n\leq \varrho(K_n)\leq 2n,
\end{align}
see \cite{DS,BobTet}. An explicit value for $\varrho(K_n)$ is not known.}\end{remark}
\begin{remark}\label{rem:1particle}
{\em  
The hypergraph version of the random walk generator \eqref{L} is given by 
%obtained by replacing the generator $\cL$ in \eqref{L} with the generator
  \begin{align}\label{Lalpha}
\cL_\a f(x)=\sum_{A\subset V: \,A\ni x}\a_{A}[\bar f_A-f(x)].
\end{align}
Note that $\cL_\a$ coincides with \eqref{L} if
  \begin{align}\label{c_alpha}
c_{xy}=\sum_{A\subset V: \,A\ni x,y}\frac{\a_{A}}{|A|}.
\end{align}
%It is interesting to compare Theorem \ref{th:meanfield1} with f
Functional inequalities such as spectral gap and (modified) log-Sobolev for this process can all be expressed by means of the Dirichlet form of the operator $\cL_\a$.
Therefore, the spectral gap inequality obtained as in  \eqref{kentalpha} by replacing $\ent f$ with $\var f$ and $\smallent_A f$ with 
$\smallvar_A f = |A|^{-1}\sum_{x\in A}(f(x)-\bar f_A)^2$, coincides with \eqref{vargap} with the choice of weights \eqref{c_alpha}. The same applies to the log-Sobolev and modified log-Sobolev when we replace the term $\smallent_A f$ in  \eqref{kentalpha} with 
$\smallvar_A \sqrt f $ and with $\smallcov_A (f,\log f)= |A|^{-1}\sum_{x\in A}(f(x)-\bar f_A)\log f(x)$ respectively. However, it is not the case for the entropy constant $\k[\a]$, that is there is no trivial way of reducing the weighted hypergraph case to the weighted graph case. The inequalities 
$\smallent_A f \leq \smallcov_A (f,\log f)$ and $\smallvar_A \sqrt f \leq \smallent_A f $ imply that the entropy constant is always between the log-Sobolev and the modified log-Sobolev constant, see Lemma \ref{lem:ele2} below. When $|A|>2$ the inequality $\smallent_{xy} f \leq 2\smallvar_{xy} \sqrt f$ has to be replaced by
\begin{align}\label{ineqents2}
\smallent_{A} f\leq \frac{\log(|A|-1)}{1-\frac2{|A|}}\,\smallvar_{A}\sqrt f,
\end{align} 
%see Lemma \ref{lem:ele2} below, 
and a significant discrepancy can occur between our entropy constant and the log-Sobolev constant in the hypergraph case when large sets are involved. 
For a concrete example, consider the mean field case $\a=\a^\ell$ for some $\ell\in\{2,\dots,n\}$. Theorem \ref{th:meanfield1} says that 
\begin{align}\label{eq:meanfield100}
\k[\a^\ell]=\frac{\binom{n-1}{\ell-1}\log \ell}{\log n}.
\end{align}
Comparing with \eqref{kKnls}, it follows that the entropy constant $\k[\a^\ell]$ is equivalent (up to a universal constant factor) to $\log \ell$ times the log-Sobolev constant and $\log \ell/\log n$ times the modified log-Sobolev constant, and thus the entropy constant interpolates between these two constants as $\ell$ grows from $2$ to $n$. }
\end{remark}

We turn to a discussion of our results for multi-particle processes. As we shall see, besides the usual tensorization properties satisfied by the (modified) log-Sobolev constants, see e.g.\ \cite{DS,Book_Toulouse,BobTet}, the entropy constant $\k$ enjoys stronger forms of tensorization. We consider two types of interacting random walk models. The first concerns independent walkers interacting through synchronous updates, while the second one can be seen as a constrained version of the first, where particles are not allowed to occupy the same vertex. In the first case the stationary distribution is a product measure, while in the second case it is uniform over permutations.

\subsection{Random walks with synchronous updates}
The synchronous updates model is defined as follows. Fix the number of particles $N\in\bbN$, and let $\O=V^N$ denote the set of vectors $\xi=(\xi_1,\dots,\xi_N)$ such that $\xi_i\in V$. Call $\nu$ the uniform distribution over $\O$, so that 
$\nu=\mu^N$ is the $N$-fold product of the uniform distribution $\mu$  over $V$. We interpret the random variable $\xi_i$ as the position of the $i-$th particle, $i=1,\dots,N$. Thus, particles are labeled. We also use the notation $\eta_A$, $A\subset V$, %to denote the unordered collection of particles 
for the set of particle labels $i$ such that $\xi_i\in A$, that is the set of particles in the block $A$. We write $\eta_z=\eta_{\{z\}}$ when the block consists of a single site. Given a collection of nonnegative weights $\a=\{\a_A,\,A\subset V\}$,  the {\em random walks with synchronous updates} on the weighted hypergraph $\a$ evolve as follows. Attach to each set $A\subset V$ independent Poisson clocks with rate $\a_A$, and when block $A$ rings all particles in $A$ simultaneously update their position by choosing independently a uniformly random position in $A$. More formally, this is the continuous time Markov chain with state space $\O$ and with infinitesimal generator 
\begin{align}\label{geninda}
\cQ_\a f=\sum_{A\subset V}\a_A(\nu_{A} f  - f)\,,
\end{align} 
where $f:\O\mapsto\bbR$ and we write  $\nu_{A}f= \nu[f|\eta_z,\,z\notin A]$ for the conditional expectation of $f$ w.r.t.\ $\nu$ given the occupation variables $\eta_z$ at all vertices $z\notin A$. The spectral gap of this process, denoted $\l[\a,N]$ is the largest $\l\geq 0$ such that
for all $f:\O\mapsto\bbR$, 
\begin{align}\label{vargapsy}
\l\,\var f \leq \sum_{A\subset V}\a_A\,\nu\left[\var_A f\right],
\end{align}
where $\var_A f = \nu_A [(f-\nu_A f)^2]$ and $\var f=\var_V f$ denotes the global variance. Similarly, the entropy constant $\k[\a,N]$ is defined as the largest $\k\geq 0$ such that
for all $f:\O\mapsto\bbR_+$,
\begin{align}\label{varentsy}
\k\,\ent f \leq \sum_{A\subset V}\a_A\,\nu\left[\ent_A f\right],
\end{align}
where $\ent_A f = \nu_A [f\log(f/\nu_A f)]$ and $\ent f=\ent_V f$ is the global entropy.  We remark that, when $N=1$, the constant $\k[\a,1]$ coincides with $\k[\a]$ defined in \eqref{kentalpha}. Indeed, in this case $f(\xi)=f(\xi_1)$, and 
\begin{align}\label{varentso}
\ent_A f = \ind_{\xi_1\in A} \smallent_A f\,,\qquad  \nu\left[\ent_A f\right]=\frac{|A|}n\,\smallent_A f.
\end{align} 
The above identities hold for the variance functional as well. Thus, reasoning as in Remark \ref{rem:1particle}, $\l[\a,1]$ coincides with $\l(G)$ where the weighted graph  $G$ is defined by \eqref{c_alpha}. 
Our result below states that this is actually the case for all $N$. 
\begin{theorem}[Independent particles with synchronous updates]\label{th:asyn}
For any weighted hypergraph $\a$, for all $N\in\bbN$,  $$\l[\a,N]=\l[\a,1]\,,\qquad \k[\a,N]=\k[\a,1].$$
\end{theorem}
\begin{remark}\label{rem:asyn}
{\em As we will see, the same proof works for both the spectral gap and the entropy constant. On the other hand it does not apply to the log-Sobolev or modified log-Sobolev constant obtained by replacing $\nu\left[\ent_A f\right]$ with $\nu\left[\var_A \sqrt f\right]$ and $\nu\left[\cov_A (f,\log f)\right]$ respectively in \eqref{varentsy}, where $\cov_A (f,g)=\nu_A[(f-\nu_A f)g]$. In fact, the independence on the number of particles does not hold in these cases in general, as can be seen e.g.\ in the simple case $\a_A=\ind_{A=V}$. In any case, convexity implies the inequality 
\begin{align}\label{conve}
\nu\left[\cov_A (f,\log f)\right]\geq \nu\left[\ent_A  f\right]
\end{align} 
for all $A$ and $f$ and thus, by Theorem \ref{th:asyn}, $\k[\a]=\k[\a,1]$ is a lower bound on the relative entropy decay of the process for all $N$. The latter, in turn, can be used to obtain new mixing time bounds for the Markov chain with generator \eqref{geninda}. In particular, using Pinsker's inequality, see e.g.\ \cite{DS}, it follows that 
the mixing time $\tmix[\a,N]$ of the process  defined by \eqref{geninda} satisfies
\begin{align}\label{pinskern}
\tmix[\a,N] \leq C\,\k[\a]^{-1}\left(\log N + \log\log|V|\right),
\end{align}
for some universal constant $C$. }
\end{remark}
\begin{remark}\label{rem:asyn2}
{\em By projection onto symmetric functions, the same independence on the number of particles holds for the spectral gap  of the unlabeled version of this process, namely when we keep track only of the occupation number of each vertex. On the other hand, by projection $\k[\a]$ is only a lower bound on the entropy constant of the unlabeled process, which could be higher. The special case when particles are unlabeled and $\a_A=0$ unless $|A|=2$ is sometimes referred to as the binomial splitting model. The latter has been recently studied in \cite{QS}, where the independence on the number of particles for the spectral gap was obtained  by a different argument.   As discussed in \cite{QS}, by duality,  controlling the convergence to equilibrium for this model allows one to control the approach to stationarity for the averaging processes introduced  in \cite{AL}.}
\end{remark}

\subsection{Block shuffles and permutations}
Next, we discuss our results for permutations. Here we have $n$ labeled particles over $n$ vertices with no two particles occupying the same vertex. Let $V$ be a vertex set with $|V|=n$ and call $\mu$ the uniform distribution over the symmetric group $\cS_n$ of the $n!$ permutations of $V$. A permutation $\si\in \cS_n$ is viewed as a vector $\si=(\si_x)_{x\in V}$ and $\si_x=i$ indicates that the particle with label $i$ is at vertex $x$.  We also use the notation $\xi_i$ to indicate the position of the particle with label $i$, that is $\xi=\si^{-1}$. Given a collection of nonnegative weights $\a=\{\a_A,\,A\subset V\}$,  we define the 
{\em $\a-$shuffle process} as the Markov chain described as follows. Attach to each set $A\subset V$ independent Poisson clocks with rate $\a_A$, and when block $A$ rings all particles in $A$ are reshuffled according to a uniform permutation of the labels currently occupying the set $A$. 
 Formally, this is the continuous time Markov chain with state space $\cS_n$ and with infinitesimal generator 
\begin{align}\label{alphashuffle}
\cG_\a f=\sum_{A\subset V}\a_A(\mu_{A} f  - f)\,,
\end{align} 
where $f:\cS_n\mapsto\bbR$ and we write  $\mu_{A}f= \mu[f|\si_z,\,z\notin A]$ for the conditional expectation of $f$ w.r.t.\ $\mu$ given the labels at all vertices $z\notin A$. The graph version, that is when $\a_A=0$ unless $|A|=2$, is known as the {\em interchange process}. 
The spectral gap of the $\a-$shuffle process, denoted $\l[\a,\cS_n]$ is the largest $\l\geq 0$ such that
for all $f:\cS_n\mapsto\bbR$, 
\begin{align}\label{vargapsn}
\l\,\var f \leq \sum_{A\subset V}\a_A\,\mu\left[\var_A f\right],
\end{align}
where $\var_A f = \mu_A [(f-\mu_A f)^2]$ and $\var f=\var_V f$ is the global variance. We remark that if we restrict to functions of $1$ particle only in \eqref{vargapsn}, then the spectral gap coincides with $\l(G)$ defined in \eqref{vargap} with the choice of rates \eqref{c_alpha}. This follows by observing that for such a function one has $f(\si)=g(\xi_1)$ for some $g:V\mapsto\bbR$ and   $$\var_A f = \ind_{\xi_1\in A} \,\smallvar_A g\,,\quad \mu\left[\var_A f\right]=\frac{|A|}n\,\smallvar_A f,$$ and by reasoning as in Remark \ref{rem:1particle}. In particular, it is always the case that $\l[\a,\cS_n]\leq \l(G)$ if $G$ is defined by \eqref{c_alpha}. For the interchange process, that is whenever $\a_A=0$ unless $|A|=2$, it is known  \cite{CLR} that $\l[\a,\cS_n]= \l(G)$. %Shortly after the publication of \cite{CLR} 
The second author of the present paper conjectured that this should be the case for arbitrary $\a$, see \cite{CesiOctopus,aldous2020life}:
 \begin{conjecture}\label{alphaconj}
For arbitrary weights $\a$,  $\l[\a,\cS_n]=\l(G)$, where $G$ is the weighted graph defined by \eqref{c_alpha}. 
\end{conjecture}
We are not aware of significant results in this direction, except for cases that can be easily reduced to the  case of graphs that was settled in \cite{CLR}. Below we show that the conjectured identity holds in the very special mean field case. Recall that $\a^\ell_A = \ind_{|A|=\ell}$. 
\begin{theorem}[Spectral gap for permutations, mean field case]\label{th:meanfieldgap}
Suppose $\a=\sum_{\ell=2}^n w_\ell\,\a^\ell$ for some nonnegative vector $w=(w_2,\dots,w_n)$. Then 
\begin{align}\label{meanfieldgap1}
\l[\a,\cS_n]=\sum_{\ell=2}^n \frac{nw_\ell}{\ell}\binom{n-2}{\ell-2},
\end{align}
and equality in \eqref{vargapsn} is uniquely achieved at mean zero functions of a single particle. 
In particular,  Conjecture \ref{alphaconj} holds whenever $\a_A$ is a function of the cardinality $|A|$ only. 
\end{theorem}

Our next result concerns the entropy constant for permutations, denoted $\k[\a,\cS_n]$. This is defined as the largest $\k\geq 0$ such that
for all $f:\cS_n\mapsto\bbR_+$,
\begin{align}\label{varentsn}
\k\,\ent f \leq \sum_{A\subset V}\a_A\,\mu\left[\ent_A f\right],
\end{align}
where $\ent_A f = \mu_A [f\log(f/\mu_A f)]$ and $\ent f=\ent_V f$ is the global entropy.  
As above, one can check that,  if we restrict to functions of $1$ particle only in \eqref{varentsn}, then the entropy constant coincides with $\k[\a]$ defined in \eqref{kentalpha}.  

Inspired by the earlier works \cite{MarStFlour,Cesi,dai2002entropy} using entropy factorization, inequalities of the form \eqref{varentsn} were recently introduced  in \cite{CaputoParisi}  in the setting of Gibbs measures describing spin systems, under the name of
{\em block factorization} of the relative entropy. These are generalizations of the classical Shearer inequality for Shannon entropy, and  play an important role in recent remarkable developments in the analysis of the convergence to equilibrium for the Glauber dynamics and related Markov chains  \cite{BCPSV,CLV20,BCCPSV}.   In particular, for  spin systems the entropy constant with mean field weights $\a=\a^\ell$ was estimated under various weak dependency assumptions in \cite{CLV20,BCCPSV}. We refer to the recent papers \cite{cryan2021modified,hermon2019modified,Anari2021entropic} for further important developments in the study of entropy inequalities under log-concavity assumptions.  
 However, we are not aware of any work concerned with the entropy constant $\k[\a,\cS_n]$ defined above. One motivation for studying this constants is the fact that a lower bound on $\k[\a,\cS_n]$ provides an upper bound on the mixing time $\tmix[\a,\cS_n]$ of the $\a$-shuffle process. Indeed, using \eqref{conve}, which continues to hold for the uniform measure on permutations, Pinsker's inequality implies 
\begin{align}\label{pinsker}
\tmix[\a,\cS_n] \leq C\,\k[\a,\cS_n]^{-1}\log n
\end{align}
for some universal constant $C$, see e.g. \cite{DS} for the well known argument. The mixing time of the interchange process is an extensively studied problem, with several interesting open questions, see \cite{Jonasson,Oliveira}. We refer to \cite{AlonKozma,HermonSalez} for recent progress in the use of functional inequalities to bound the mixing time of the interchange process for certain sequences of graphs. Furthermore, the mixing time of hypergraph versions has been recently investigated in \cite{connor2019mixing,hermon2021direct}.  

Our main result concerning the entropy constant for permutations is a computation of its value in all mean field cases.
\begin{theorem}[Entropy constant for permutations, mean field case]\label{th:meanfieldent}
Suppose $\a=\sum_{\ell=2}^n w_\ell\,\a^\ell$ for some nonnegative vector $w=(w_2,\dots,w_n)$. Then 
\begin{align}\label{th:meanfield2}
\k[\a,\cS_n]= \sum_\ell w_\ell\, \frac{\binom{n}\ell \log \ell !}{\log n!},
\end{align} 
%\begin{align}\label{meanfieldent1}
%\k[\a,\cS_n]=\sum_{\ell=2}^n 
%\frac{nw_\ell}{\ell}\binom{n-2}{\ell-2},
%\end{align}
and equality in \eqref{varentsn} is uniquely achieved at multiples of a Dirac mass. 
\end{theorem}
\begin{remark}\label{rem:entperm}
{\em We observe that the phenomenon of independence on the number of particles for the spectral gap in  Conjecture \ref{alphaconj}  cannot hold for the entropy constant. Indeed, if e.g.\ $\a=\a^\ell$ we know from Theorem  \ref{th:meanfieldent} and Theorem  \ref{th:meanfield1} that
\begin{align}\label{ent1meanfield2}
\k[\a^\ell]=\frac{\binom{n-1}{\ell-1}\log \ell}{\log n} > \frac{\binom{n}\ell \log \ell !}{\log n!}=\k[\a^\ell,\cS_n].
\end{align}
For instance, when $\ell=2$, the left hand side above is asymptotically twice as large as the left hand side as $n\to\infty$. On the other hand the ratio approaches $1$ for $n\to\infty$ and then $\ell\to\infty$. }
\end{remark}

\begin{remark}\label{rem:entperm2}
{\em 
We prove Theorem \ref{th:meanfieldent} by a suitable version of the martingale method already employed in the estimation of the log-Sobolev and modified log-Sobolev constants for the interchange process on the complete graph \cite{LeeYau,GaoQuastel,Goel}. It is remarkable that in those cases, which correspond to the $\ell=2$ case of the above theorem, the method does not allow one to compute exactly the constants but only to give an estimate that is tight up to a constant factor. For the entropy constant instead one can provide an explicit value and a characterization of the extremal functions associated to it. As we will see in Corollary \ref{cor:permanent} below, the explicit knowledge of the entropy constant can be quite  useful.

}
\end{remark}

A further result concerns the case of unlabeled particles, namely when there are $r$ unlabeled particles undergoing the $\a$-shuffle dynamics, for some $r\in\{1,\dots,n-1\}$. This amounts to restricting the action of the generator \eqref{alphashuffle}
to functions  of the form $f(\si)=g(\xi_1,\dots,\xi_r)$ for some symmetric function $g:V^r\mapsto\bbR$. The stationary distribution becomes the uniform measure over all $\binom{n}{r}$ configurations.  In the binary mean field case $\a_A=\ind_{|A|=2}$, this is known as the {\em Bernoulli-Laplace model} \cite{diaconis1987time}.  The log-Sobolev constant for this process was estimated in \cite{LeeYau}, while its modified log-Sobolev constant was estimated in \cite{GaoQuastel,Goel,ErbarMaasTetali}.  As in the labeled case, these estimates are tight up to constant factors but the exact value of the constants remains unknown. Here we are able to compute the corresponding entropy constant, denoted $\k(n,r)$, which is obtained by restricting \eqref{varentsn}
to the above described class of functions in the case $\a_A=\ind_{|A|=2}$. 
\begin{theorem}[Entropy constant for Bernoulli-Laplace]\label{th:BL}
For all integers $n\geq 2$ and all $1\leq r\leq n-1$, the entropy constant $\k(n,r)$
of  the Bernoulli-Laplace model with $r$ particles satisfies
\begin{align}\label{meanfieldent3}
\k(n,r)=\frac{r(n-r)\log(2)}{\log\binom{n}r}.
\end{align} 
%and the inequality is saturated by a Dirac mass. 
\end{theorem}
\begin{remark}\label{rem:entperm3}
{\em 
Clearly, the case $r=1$ coincides with the case $\ell=2$ of Theorem \ref{th:meanfield1}, that is $\k(n,1)=\k[\a^2]=(n-1)\log (2)/\log n $. On the other hand for $r\sim n/2$ one has $\k(n,r)\sim n/4 $. This is in contrast with the spectral gap, which is independent of $r$, and with the modified log-Sobolev constant which is known to be equivalent up to a factor $4$ for all values of $r$ \cite{Goel}. The inequalities \eqref{ineqents} on the other hand show that the log-Sobolev constant is equivalent up to a factor $\log(2)$ to the constant $\k(n,r)$, which implies a slight refinement of the %and indeed an equivalent  behavior to \eqref{meanfieldent3} was shown to hold 
estimates in \cite[Theorem 5]{LeeYau}. 
}
\end{remark}
Finally, as an application of our results we mention the following sharp  upper bound on the permanent of a matrix with arbitrary nonnegative entries, which was independently conjectured by the second author, by Carlen, Lieb, Loss \cite{CLL2} and by Samorodnitsky \cite{Samorodnitsky}. 
Let $A=(a_{i,j})$ denote an $n\times n$ matrix, and write 
 \begin{align}\label{perma1}
{\rm perm}(A) = \sum_{\si\in \cS_n}\prod_{i=1}^n a_{i,\si_i},
 \end{align} 
 for the permanent of $A$. For zero-one valued matrices, the well known Bregman-Minc theorem \cite{Bregman} establishes a tight upper bound on the permanent of a matrix with given row sums. Our result below can be seen as an extension of the Bregman-Minc theorem to all matrices with nonnegative entries. 
 
 \begin{corollary}\label{cor:permanent}
 For any $p\geq 1$, for any $n\times n$ nonnegative matrix $A$, 
 \begin{align}\label{perma11}
{\rm perm}(A) \leq \max\left\{1,\frac{n!}{n^{n/p}}\right\} \prod_{i=1}^n \|R_i\|_p,
 \end{align} 
 where $R_i$ denotes the $i-$th row of $A$ and $\|\cdot\|_p$ denotes the $\ell_p$-norm of a vector, and equality is uniquely achieved at either the identity matrix or  the  all\,-\,$1$ matrix (up to permutation of rows and multiplication by a scalar). 
 \end{corollary} 
Corollary \ref{cor:permanent} proves Conjecture 1.1 in \cite{Samorodnitsky}. Note that the values $1$ and $\frac{n!}{n^{n/p}}$ correspond to the case where $A$ is the identity matrix or $A$ is the all-$1$ matrix respectively, and thus \eqref{perma11} is optimal. The bound was shown to hold in \cite{CLL2} when $p\geq 2$. It was also proved in \cite{Samorodnitsky} for $p\in(1,2)$ with an extra factor growing subexponentially with $n$. As shown in \cite[Lemma 1]{Samorodnitsky}, the statement \eqref{perma11} can be reduced to proving the bound 
\begin{align}\label{perma2}
{\rm perm}(A) \leq \prod_{i=1}^n \|R_i\|_{p_c},\qquad p_c= \frac{n\log n}{\log(n!)},
 \end{align} 
where $p_c$ is the value at which the increasing function $p\mapsto \frac{n!}{n^{n/p}}$ takes the value $1$. We will see that  this estimate follows from the case $\ell=n-1$ of our Theorem \ref{th:meanfield2}, which implies a sharp subadditivity estimate for the entropy functional, see Lemma \ref{lem:equiv} below.
The use of entropy to prove upper bounds on the permanent goes back to \cite{schrijver1978short,radhakrishnan1997entropy}. We refer to \cite{gurvits2014bounds,anari2021tight} for further generalizations of the Bregman-Minc theorem.
 
\subsection{Miscellaneous remarks}
We end this introduction with a few remarks on open problems and conjectures.
It would be nice to compute the constant $\k(G)$ defined in \eqref{kent} for various classes of graphs. Besides the complete graph $c_{xy}\equiv 1$ which is contained in Theorem \ref{th:meanfield1}, see Remark \ref{rem:K_n}, the determination of $\k(G)$ remains in general a difficult problem, much as in the case of the (modified) log-Sobolev constants. 
In Proposition \ref{prop:star} below we consider the star graph $c_{xy}=\ind_{x=x_0}$ for some fixed vertex $x_0\in V$, for which we determine the entropy constant asymptotically.  An interesting question is to determine for which graphs one should have $\k(G)=\l(G)$. For the log-Sobolev constant it was shown \cite{ChenSheu} that this is the case for even cycles; see also \cite{chen2008logarithmic}. We believe that the same holds for the entropy constant. Moreover, we believe that the identity $\k(G)=\l(G)$ could extend %(in contrast to the log-Sobolev case) 
to all $n$-cycles with $n>3$, and to all paths, that is graphs defined by $c_{xy} = \ind_{|x-y|=1}$, $x,y\in\{1,\dots,n\}$, and more generally for rectangular boxes in $\bbZ^d$, $d\geq 2$. As we point out in Remark \ref{rem:kala} below, if a graph $G$ satisfies the identity $\k(G)=\l(G)$ then its (modified) log-Sobolev constants are also determined in terms of the spectral gap. 

A very interesting question in the setting of permutations is to determine for which hypergraphs one has $\k[\a,\cS_n]=\k[\a]$. As we have seen in Remark \ref{rem:entperm} this cannot hold for the mean field case. However, this could be the case for  certain specific graphs such as the path or more generally for rectangular boxes in $\bbZ^d$, $d\geq 2$. As a consequence of \eqref{pinsker}, that would allow one to obtain sharp mixing time bounds for the interchange process on such graphs. Together with the entropic characterization of the cutoff phenomenon recently developed by Salez \cite{Salez}, that may even provide a direct proof of the cutoff phenomenon for these processes, thus generalizing Lacoin's cutoff result for the path \cite{Lacoin}. 
In this respect, it may be of interest to investigate the validity of the entropic analogue of the {\em octopus inequality} established in \cite{CLR} for the interchange process, namely the following inequality for all weighted graphs, for all $x\in V$, and for all functions $f:\cS_n\mapsto\bbR_+$:  
\begin{align}\label{entoctopus}
\frac12\sum_{y,z\neq x}c^{*,x}_{yz}\mu\left[\ent_{yz} f\right]\leq 
\sum_{y} c_{xy}\mu\left[\ent_{xy} f\right],
\end{align} 
where $c^{*,x}_{yz} = c_{xy}c_{xz}/\sum_{w} c_{xw}$. This inequality is known to hold when the entropy functional is replaced by the variance functional, see \cite[Theorem 2.3]{CLR}. Since, as in Lemma \ref{lem:ele}, 
\begin{equation}\label{ineqents1}
2\log(2)\var_{yz} \sqrt f \leq \ent_{yz} f \leq 2\var_{yz} \sqrt f,
\end{equation} for all $f:\cS_n\mapsto\bbR_+$ and all $y,z$, we know that \eqref{entoctopus} holds with at most an extra factor $1/\log(2)$ in the right hand side.  
The validity of \eqref{entoctopus} would have interesting consequences for the analysis of mixing times, see \cite{AlonKozma} for a thorough discussion of this in the case of the variance functional. Moreover, it would encourage the use of a recursive approach to the analysis of the entropic constant based on electric network reductions, in analogy with the main argument in \cite{CLR}. However, as we discuss in Section \ref{sec:1} below, it should be pointed out that, in contrast with the spectral gap and (modified) log-Sobolev constants, the entropic constant $\k(G)$ does not always satisfy the simple monotonicity $\k(G_x)\geq \k(G)$ if $G_x$ denotes the graph $G$ after the electric network reduction at node  $x$.

\bigskip

\noindent
{\bf Acknowledgements:} We would like to thank Justin Salez for  helpful conversations around the topics of this work. A.\ Bristiel would like to thank the Dipartimento di Matematica e Fisica of Roma Tre for its warm welcome in this period of crisis.

\section{One particle problems}\label{sec:1}
We start by recalling the definition of the log-Sobolev and modified log-Sobolev constants and their relations with the entropy constant $\k(G)$ defined in \eqref{kent}. We then prove some preliminary technical estimates that will be used in the proof of our main results. Next, we prove Theorem \ref{th:meanfield1}. The section ends with  the entropy constant for a star graph, and with a discussion of the behavior of entropy constants under electric network reductions. 

\subsection{Entropy constant vs.\ log-Sobolev and modified log-Sobolev}
For a weighted graph $G$ with vertex set $V$ with $|V|=n$, and edge weights $c_{xy}$,  
 the log-Sobolev constant $\b(G)$, and the modified log-Sobolev constant $\varrho(G)$ are defined, respectively as the largest $\b,\varrho$ such that for all $f:V\mapsto\bbR_+$
 \begin{align}\label{kentbr}
\b\,\ent f \leq \frac2{n} \sum_{x,y\in V}c_{xy}\,\smallvar_{xy} \sqrt f,\qquad \varrho\,\ent f \leq \frac2{n} \sum_{x,y\in V}c_{xy}\,\smallcov_{xy} (f,\log f),
\end{align}
where $\ent f = \mu\left[f\log(f/\mu(f))\right]$ denotes the relative entropy of $f$ with respect to the uniform distribution $\mu$ on $V$ and we use the notation
 \begin{gather}\label{entbr2} 
\smallvar_{xy} f = \frac14(f(x)-f(y))^2\,,
%\qquad  u(a,b) = \frac14(a-b)^2\,,\\
\quad \smallcov_{xy} (f,\log f)=\frac14(f(x)-f(y))\log\frac{f(x)}{f(y)}.
%v(f(x),f(y))\,,\qquad  v(a,b) = \frac14(a-b)\log(a/b).
\end{gather}
Let $\l(G),\k(G)$ be defined as in \eqref{vargap} and \eqref{kent} respectively. 
\begin{lemma}\label{lem:ele}
For any weighted graph $G$, 
\begin{align}\label{kbr1}
2\log(2)\b(G)\leq \k(G)\leq 2\b(G)\leq \frac12\varrho(G)\leq \l(G).
\end{align}
\end{lemma}
\begin{proof}
We are going to observe that for all edges $xy$, for all $f:V\mapsto\bbR_+$,
 \begin{align}\label{kbr2}
2\log(2)\smallvar_{xy} \sqrt f\leq \smallent_{xy} f \leq 2\,\smallvar_{xy} \sqrt f \leq 
\frac12\,\smallcov_{xy} (f,\log f).
\end{align}
The relations \eqref{kbr1} follow all from \eqref{kbr2} except from the inequality $\frac12\varrho(G)\leq \l(G)$. This is however a well known estimate following from linearization, see e.g.\ \cite{BobTet}. 
The bounds $\smallent_{xy} f \leq 2\,\smallvar_{xy}\sqrt f$ and  $2\,\smallvar_{xy} \sqrt f \leq 
\frac12\,\smallcov_{xy} (f,\log f)$ can be found in \cite[Theorem A.2]{DS} and \cite[Lemma 2.7]{DS} respectively. It remains to prove the first inequality in \eqref{kbr2}. To this end,  we define
\begin{align}\label{def:psi}
\psi(s,t)=\frac12\,s\log(s) + \frac12\,t\log(t) - \frac{s+t}2\log\frac{s+t}2,
\end{align}
and $u(s,t) = \frac14(\sqrt s-\sqrt t)^2$. %,  $v(s,t) = \frac14(s-t)\log(s/t)$. 
By homogeneity and symmetry it is sufficient to prove that 
	\begin{align}\label{psiuv}
		2\log(2)u(\cos^2(\theta),\sin^2(\theta))\leq \psi(\cos^2(\theta),\sin^2(\theta)),\qquad \theta\in[0,\pi/4].
	\end{align}
	This choice of parametrization makes calculations more straightforward. Consider the function $g(\theta)=\psi(\cos^2(\theta),\sin^2(\theta))/u(\cos^2(\theta),\sin^2(\theta))$, $\theta\in[0,\pi/4]$, with the value at $\theta=\pi/4$ defined by continuity. Since $g(0)=2\log(2)$, it suffices to show that $g'(\theta)\geq 0$ for $\theta\in[0,\pi/4]$. A computation shows that
	\begin{equation*}
		g'(\theta) = 4\,\frac{\cos(\theta)\log(2\cos^2(\theta))+\sin(\theta)\log(2\sin^2(\theta))}{(\cos(\theta)-\sin(\theta))^3}.
	\end{equation*}
		By elementary differentiation one can check that the application $$x\mapsto \sqrt{x}\log(2x)+\sqrt{1-x}\log(2(1-x)),\qquad  x\in [0,1],$$ is convex and has a minimum at $x=1/2$ where it takes the value $0$. So the numerator, $\cos(\theta)\log(2\cos(\theta)^2)+\sin(\theta)\log(2\sin(\theta)^2)$ is non-negative. 	Since  $\cos(\theta)\geq\sin(\theta)$ for $\theta\in [0,\pi/4]$ this ends the proof. 
%
%	Notice that, by convexity of $\varphi : x\mapsto \sqrt{x}\log(x)$,
%	\begin{equation*}
%		\cos(\theta)\log(2\cos(\theta)^2)+\sin(\theta)\log(2\sin(\theta)^2)\geq \sqrt{2}\varphi\PAR{\frac{2\cos^2(\theta)+2\sin^2(\theta)}2} = 0.
%	\end{equation*} 
%	Since  $\cos(\theta)\geq\sin(\theta)$ for $\theta\in [0,\pi/4]$ this ends the proof. 
	%and concluding the proof of the lemma.
	%	By homogeneity it is sufficient to prove that 
%\begin{align}\label{psiuv}
%2\log(2)u(t,1-t)\leq \psi(t,1-t),\qquad t\in[0,1].
%\end{align}
%Consider the symmetric function $g(t)=u(t,1-t)/\psi(t,1-t)$, $t\in[0,1]$, with the value at $t=1/2$ defined by continuity. Since $g(1/2)=1/2$ and $g(1)=g(0)=1/(2\log(2)) > 1/2$, it suffices to show that $g'(t)\geq 0$ for $t\in[1/2,1]$. A computation shows that
%\begin{gather*}
%g'(t)= U(t)V(t)\,,\qquad U(t)=\frac{\sqrt t - \sqrt {1-t}}{2(\log(2) +t\log t+(1-t)\log(1-t))}\,,\\
%V(t)= \left(\frac1{\sqrt{t}} + \frac1{\sqrt{1-t}}\right) - \frac12\left(\sqrt t - \sqrt {1-t}\right)\log\frac{t}{1-t}. 
%%((1/(2 Sqrt[1 - t]) + 1/(2 Sqrt[t])) (-Sqrt[1 - t] + Sqrt[
%%   t]))/(2 (Log[2]/2 + 1/2 ((1 - t) Log[1 - t] + t Log[t]))) 
%\end{gather*}
%Since $U(t)\geq 0$ for $t\in[1/2,1]$, it suffices to show that $V(t)\geq 0$ for all $t\in[0,1]$. The latter can be shown by direct inspection. 
%\what 
%{\em I could not find an elementary proof that $V(t)\geq 0$, $t\in[0,1]$ can you see one ? one may also try to prove the more general inequaliy: $a^{-1}+b^{-1} - (a-b)\log(a/b)$ for all $a,b\in[0,1]$ which seems to be true.} 
\end{proof}
\begin{remark}\label{rem:kala}
{\em An immediate consequence of Lemma \ref{lem:ele} is that a weighted graph $G$ such that $\k(G)=\l(G)$ must satisfy $2\b(G)=\l(G)=\frac12\varrho(G)$.}
\end{remark}
\begin{remark}\label{rem:kala2}
{\em Another simple consequence of Lemma \ref{lem:ele} is that for all weighted graphs $G$ with $n>2$ nodes,
\begin{align}\label{kala3}
\k(G)\geq \frac{\left(1-\tfrac2n\right)2\log(2)}{\log(n-1)}\,\l(G).
\end{align}
Indeed, for every $f\geq 0$, \cite[Corollary A.4]{DS} shows that 
\begin{align}\label{kala4}
\ent f \leq \frac{\log(n-1)}{\left(1-\tfrac2n\right)} \,\var \sqrt f.
\end{align} 
Therefore, using  the definition of $\l(G)$ and \eqref{kbr2}, 
$$
\frac{(1-\tfrac2n)}{\log(n-1)}\,\l(G) \,\ent f \leq
\frac{2}n\sum_{x,y}c_{xy} \smallvar_{xy} \sqrt f \leq \frac1{2\log(2)}\,\frac{2}n\sum_{x,y}c_{xy} \smallent_{xy} f. 
$$
}
\end{remark}
In the hypergraph setting the comparison is not as tight as in Lemma \ref{lem:ele}. Let $\a=\{\a_A,\,A\subset V\}$ denote a collection of nonnegative weights, and let $\k[\a]$ be defined as in \eqref{kentalpha}. Recall the definitions 
\begin{align}\label{smallvarcov}
\smallvar_A f =\frac1{|A|}\sum_{x\in A}(f(x)-\bar f_A)^2,  \quad
\smallcov_A (f,\log f)= \frac1{|A|}\sum_{x\in A}(f(x)-\bar f_A)\log f(x).
\end{align}
and let $\b[\a],\varrho[\a]$ be defined as in \eqref{kentalpha} with $\smallent_A f$ replaced by $\smallvar_A \sqrt f$ and $ \smallcov_A (f,\log f)$ respectively. \begin{lemma}\label{lem:ele2}
For any $A\subset V$, for all $f:V\mapsto\bbR_+$,
\begin{align}\label{kbr10}
\smallvar_A \sqrt f \leq \smallent_A f \leq \smallcov_A (f,\log f) .
\end{align}
In particular, for any $\a$ one has
\begin{align}\label{kbr11}
\b(G)=\b[\a]\leq \k[\a]\leq \varrho[\a]=\varrho(G)\,,
\end{align}
where $G$ is given by \eqref{c_alpha}. 
%Moreover, for any $|A|>2$,
%\begin{align}\label{ineqents200}
%\smallent_{A} f\leq \frac{\log(|A|-1)}{1-\frac2{|A|}}\,\smallvar_{A}\sqrt f.
%\end{align} 
\end{lemma}
\begin{proof}
The inequality $\smallvar_A \sqrt f \leq \smallent_A f$ is a consequence of \cite[Lemma 1]{LO}, while $\smallent_A f \leq \smallcov_A (f,\log f)$
follows from Jensen's inequality, since 
$$
\smallcov_A (f,\log f) =  \frac1{|A|}\sum_{x\in A}f(x)\log f(x) - \frac{\bar f_A}{|A|}\sum_{x\in A}\log f(x)\geq \smallent_A f .
$$
The relations $\b[\a]\leq \k[\a]\leq \varrho[\a]$ follow from \eqref{kbr10}. %By Remark \ref{rem:1particle} this implies \eqref{kbr11}.
As observed in Remark \ref{rem:1particle} we have $\b[\a]=\b(G)$ and $\varrho[\a]=\varrho(G)$ if $G$ is given by \eqref{c_alpha}.
%The bound \eqref{ineqents200} follows from \cite[Theorem A.1]{DS}. 
\end{proof}

\subsection{Technical estimates}
Define $\mu_\r(a,b) = \r a+(1-\r)b$ the mean of $(a,b)$ w.r.t $\cB(\r)$, the Bernoulli distribution with parameter $\r$. The entropy of $(a,b)$ w.r.t $\cB(\r)$  will be written as
\begin{equation}\label{def:psir}
	\psi_\r(a,b) = \r \,a\log(a/\mu_\r(a,b)) +(1-\r)b\log(b/\mu_\r(a,b)).
\end{equation}
We define the symmetrized entropy
\begin{equation*}
	\bar\psi_\r(a,b) = \frac{\psi_\r(a,b)+\psi_{1-\r}(a,b)}2 =  \frac{\psi_\r(a,b)+\psi_\r(b,a)}2,
\end{equation*}
and let $h(\r) = -\r\log(\r)-(1-\r)\log(1-\r)$  denote the Shannon entropy of $\cB(\r)$.
\begin{lemma}\label{lem:tec1}
	For any $a,b\geq 0$ the application
	\begin{equation*}
		\r\mapsto\frac{\bar\psi_\r(a,b)}{h(\r)},
	\end{equation*}
	is non-decreasing for $\r\in[0,1/2]$.
\end{lemma}
\begin{proof}
	By homogeneity we can assume $a=t,b=1-t$, with $t\in[0,1]$. To avoid heavy notation we will also write $\mu_\r(t,1-t)$ as $\mu_\r$.
	After simplifications,
	\begin{equation*}
		\frac{\partial}{\partial\r}\frac{\bar\psi_\r(t,1-t)}{h(\r)} =\frac1{h(\r)^2}\PAR{ h(\r)\frac{1-2t}2\log\PAR{\frac{\mu_\r}{\mu_{1-\r}}}-\bar\psi_\r(t,1-t)\log\PAR{\frac {1-\r}\r}}.
	\end{equation*}
	Fix now $\r\in[0,1/2]$ and define
	\begin{equation*}
		g(t) =  h(\r)\frac{1-2t}2\log\PAR{\frac{\mu_\r}{\mu_{1-\r}}}-\bar\psi_\r(t,1-t)\log\PAR{\frac {1-\r}\r},
	\end{equation*} so that $\frac{\partial}{\partial\r}\frac{\bar\psi_\r(t,1-t)}{h(\r)}$ has the same sign as $g$. We have to show that $g(t)\geq 0$, $t\in[0,1]$.
	
	\begin{figure}[h!]
		\centering
		\begin{subfigure}[t]{0.5\textwidth}
			\centering
			\includegraphics[height=2in]{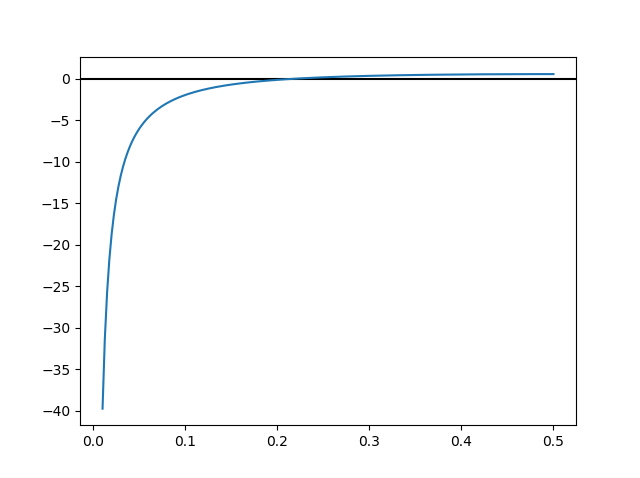}
			\caption{Graph of $g''$ for fixed $\r=0.3$}
		\end{subfigure}%
		~ 
		\begin{subfigure}[t]{0.5\textwidth}
			\centering
			\includegraphics[height=2in]{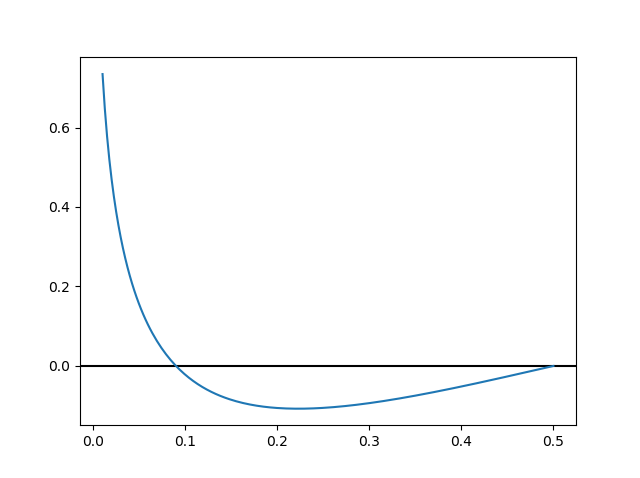}
			\caption{Graph of $g'$ for fixed $\r=0.3$}
		\end{subfigure}
		~
		\begin{subfigure}[t]{0.5\textwidth}
			\centering
			\includegraphics[height=2in]{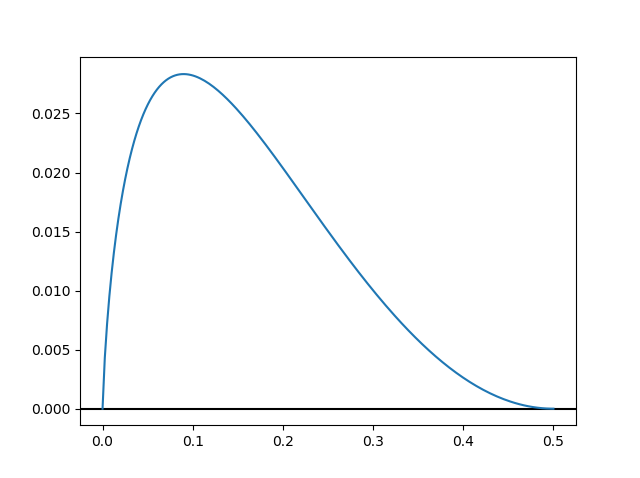}
			\caption{Graph of $g$ for fixed $\r=0.3$}
		\end{subfigure}
	\caption{Variations of the function $g$ from the proof of Lemma \ref{lem:tec1}}
	\label{fig1}
	\end{figure}
	By differentiating twice, %again and after simplifications :
	\begin{equation*}
	\begin{split}
		g''(t)& = h(\r)\frac{2(1-2\r)}{\mu_\r\mu_{1-\r}}+h(\r)\frac{(1-2t)^2(1-2\r)^2}{2(\mu_\r\mu_{1-\r})^2}-\log\PAR{\frac {1-\r}\r}\frac{\r(1-\r)}{2t(1-t)\mu_\r\mu_{1-\r}}\\
		&=
		\frac{P(t)}{2t(1-t)(\mu_\r\mu_{1-\r})^2}	,
	\end{split}
	\end{equation*}
	where $P$ is the polynomial  of degree 4 defined by
	\begin{align*}
	P(t)&=4(1-2\r)h(\r)t(1-t)\mu_\r\mu_{1-\r}+\\& \qquad+h(\r)t(1-t)(1-2t)^2(1-2\r)^2 -\log\left({\frac {1-\r}\r}\right)\r(1-\r)\mu_\r\mu_{1-\r}.
	\end{align*}
	%The numerator is a polynomial in $t$ of degree 4, denote it $P$. 
	Thus, $g''$ has the same sign as $P$.
		Notice $P(1-t)=P(t)$, thus $P(x) = Q(x(1-x))$ for some polynomial $Q$ of degree 2. In fact if $Q(x)=ax^2+bx+c$, then $a=-8\r(1-2\r)^2h(\r) < 0$ and  $b > -2a$ for $\r\leq1/2$. Therefore $Q$ is  increasing for $x\in[0,1/4]\subset(-\infty,-b/2a]$. Furthermore $t\mapsto t(1-t)$ is an increasing bijection from $[0,1/2]$ to $[0,1/4]$, thus $t\mapsto P(t)=Q(t(1-t))$ is increasing for $t\leq 1/2$. In particular, $P$ has at most one root $x$ in $[0,1/2]$. 
	In conclusion,  $g''$ has at most one root in $[0,1/2]$ and $g''(0^+)=-\infty$ so $g$ has the behavior depicted in Figure \ref{fig1}.

	Suppose there exists a  local minimum $t_0 \in(0, 1/2)$ of $g$, so that  $g'(t_0) = 0$ and $g''(t_0)\geq 0$. Since $g''$ changes  sign  at most once, we have $g''(t)\geq 0$ for all $t\in[t_0,1/2]$. But one can check that $g'(1/2) = 0$, so in that case $g'$ would be constant on the interval $[t_0,1/2]$. It follows that $P=0$ on $[t_0,1/2]$ and since this interval has a non-empty interior, $P=0$ uniformly. So $g''=g'=0$ on $[0,1/2]$, and $g(1/2) = 0$ so $g$ is also constant at 0, proving that $g$ is non-negative.
	
	On the other hand, if $g$ is non-constant, then the only stationary point of $g$ with non-negative second derivative is $1/2$, proving that the minimum of $g$ is either at $0$, or $1/2$. 
	However, $g(1/2)=g(0)=0$ 
	and thus  	$g(t) \geq 0$, for all $t\in[0,1]$.
		\end{proof}

In the case $\r=1-\r=1/2$ we write $\psi=\psi_{1/2}=\bar\psi_{1/2}$ as in \eqref{def:psi}. Thus Lemma \ref{lem:tec1} shows that 
$\bar\psi_\r(a,b)/h(\r)\leq \psi(a,b)/h(1/2)$ for all $\r\in[0,1]$ and all $a,b\geq 0$.
%Since the application defined in Lemma \ref{lem:tec1} is constant for $a=0,b=1$, we immediately have the following corollary.
The following estimate is then a direct corollary of Lemma \ref{lem:tec1}. 
\begin{corollary}\label{cor:tec1}
	For any $\r\in [0,1]$,
	\begin{equation*}
		\sup_{a,b >0}\frac{\bar\psi_\r(a,b)}{\psi(a,b)} = \frac{\bar\psi_\r(0,1)}{\psi(0,1)}=\frac{h(\r)}{\log(2)}.
	\end{equation*}
	%and the maximum is achieved at $(a,b)=(0,1)$ or $(a,b)=(1,0)$.
\end{corollary}
%\begin{remark}
%	Adapting the proof for the non-symmetric case we can get the same estimate but this time $h$ is replaced by a function $f$ such that $\psi_\r(1,0)/f(\r),\psi_\r(0,1)/f(\r)$ have non-negative derivatives in $\r$. Meaning we can take:
%	\begin{equation*}
%		f(\r) =-\r\log(\r).
%	\end{equation*}
%\end{remark}

We will also need the following higher dimensional estimate. For any $a\in\bbR_+^n$, we define
	\begin{equation}\label{def:psihat}
		\widehat\psi_\rho(a) = \frac1n\sum_{i=1}^n\psi_\rho(a_i,\bar a_i),
	\end{equation}
	with $\bar a_i = \frac1{n-1}\sum_{j\neq i}a_j$, and $\ent(a)=\frac1n\sum_{i=1}^n\log(a_i/\mu(a))$, with $\mu(a)=\frac1n\sum_{i=1}^na_i$.

\begin{lemma}\label{lem:tec2}
	For any $n\geq 2$ and for $\rho= 1/n$, 
	\begin{equation}\label{eq:monf}
		\sup_{a\in\R_+^n}\frac{\widehat\psi_\rho(a)}{\Ent(a)}  = \frac{h(1/n)}{\log n},
	\end{equation}
	where the supremum is over all non-constant $a\in\R_+^n$.
\end{lemma}
\begin{proof}
	We may assume $n>2$ since the claim is trivial at $n=2$. We observe that when $a$ is a Dirac mass $a_i=\IND_{i=i_0}$ for some fixed $i_0$, then $\sum_{i=1}^n\psi_{1/n}(a_i,\bar a_i)=h(1/n)$ and $\Ent(a)=\log(n)/n$,  and thus we need to show that the Dirac mass is a maximizing vector for $\widehat\psi_{1/n}(\cdot)/\Ent(\cdot)$. To achieve our goal we will start by proving that any maximizing function takes only two values, and then that in fact it takes only once its maximal value and $n-1$ times its other value. Finally, we will observe that the maximum of $\widehat\psi_{1/n}(\cdot)/\Ent(\cdot)$ over such functions is achieved only at Dirac masses.
	
	Let  us start by proving the inequality when $a$ approaches a constant vector. If $a = 1+\varepsilon g$ with $\mu[g] = 0$, and $\varepsilon>0$, then we can expand % $\widehat\psi_\rho$ asymptotically :
	\begin{equation*}
		\begin{split}
			\widehat\psi_\rho(a) &= \frac{\varepsilon^2}2\rho(1-\rho)\frac1n\sum_{i=1}^n(g_i-\bar g_i)^2+o(\varepsilon^2)\\
			& = \frac{\varepsilon^2}2\rho(1-\rho)\frac1n\sum_{i=1}^n((1+1/(n-1))g_i)^2+o(\varepsilon^2)\\
			&= \frac{\varepsilon^2\rho(1-\rho)n^2}{2(n-1)^2}\var(g)+o(\varepsilon^2).
		\end{split}
	\end{equation*}
	Since $\ent(a) = \frac{\varepsilon^2}2\var(g)+o(\varepsilon^2)$, it follows that $\widehat\psi_{1/n}(a)/\Ent(a)\to 1/(n-1)$ as $\varepsilon\to 0$, for any fixed $g$ with $\mu(g)=0$.  One can check that, for all $n>2$,
	\begin{equation*}
		\frac{1}{n-1} < \frac{h(1/n)}{\log(n)} <\frac2n.
	\end{equation*}
	In particular, this proves that the supremum of $\widehat\psi_{1/n}(a)/\Ent(a)$ is not at a point of discontinuity and so there exists a non-constant maximizing vector $a$ at which the extremal value is achieved. 
	
	Since $\rho =1/n$, one has $\rho a_i+(1-\rho)\bar a_i = \mu(a)$. Thus,  assuming $\mu(a)= 1$,
	\begin{equation*}
		\partial_i\widehat\psi_\rho(a) = \frac1n\rho \log(a_i)+\frac{1-\rho}{n(n-1)}\sum_{j\neq i}\log(\bar a_j).
	\end{equation*}
	Let $K\geq K_*:=h(1/n)/\log(n)$ be the maximum of $\widehat\psi_\rho(\cdot)/\Ent(\cdot)$. 
	We will start by proving that any maximizing function $a:V\mapsto \bbR_+$ is two valued. % using the variational principle. \\
	If $a=(a_i)$ is a maximizing function with $\mu(a)= 1$, then we must  have
		\begin{equation*}
		\grad[\widehat\psi_\rho(x)-K\Ent(x)]_{|x=a} = 0.
	\end{equation*}
	In other words,  $a$ must verify
	\begin{equation}\label{eq:zerograd}
		\rho \log(a_i)+\frac{1-\rho}{n-1}\sum_{j\neq i}\log(\bar a_j) = K\log(a_i),
	\end{equation}
	for each $i=1,\dots,n$. 
	Since $\mu(a)=1$, we have $\bar a_i = (n-a_i)/(n-1)$. Therefore each coordinate $a_i$ follows the same equation:
	\begin{equation}\label{eq:grad}
		(K-\rho)\log(x)+\rho\log\PAR{\frac{n-x}{n-1}} = A, 
	\end{equation}
	where $A=\rho\sum_i\log(\bar a_i)$. %and using $1-\rho=(n-1)\rho$. 
	The left hand side above %is last equation 
	is increasing for $x\in[0,n-1/K]$ and decreasing on $[n-1/K,n]$. Therefore, for any $A$, the equation has at most 2 solutions. This shows that the maximiser $a$ is at most two valued. We write $x_1<x_2$ for these two values. 
	%Although $A$ depends on $a$, the function $x\mapsto (K-\rho)\log(x)+\rho\log\PAR{\frac{n-x}{n-1}}$ does not and is the same for each coordinate of $a$, so $a$ cannot take more than 2 different values.
	
	Let  $k$ be the number of times $a$ takes the value $x_2$. Since $\mu(a)=1$ we have $x_1 = (n-kx_2)/(n-k)$. The condition $x_2 > x_1\geq 0$ is equivalent to $n/k\geq x_2 >1$. Thus, going back to the initial equation \eqref{eq:zerograd}, $x_2$ is a solution of
	\begin{equation}\label{eq:minf}
		f(x) = (K-\rho)\log(x)-\rho\PAR{(k-1)\log\PAR{\frac{n-x}{n-1}}+(n-k)\log\PAR{\frac{n-\bar x}{n-1}}} =0,
	\end{equation}
in the interval $(1,n/k]$, 	where $\bar x:=(n-kx)/(n-k)$.

	First, we show that if $k\geq 2$ then there is no solution of \eqref{eq:minf}  in the domain $(1,n/k]$. Since $f(1)=0$, it will be sufficient to show that $f$ is strictly increasing for $x\in[0,n/k]$.
	For $k\geq 2$, after factorization  we have
	\begin{equation*}
		f'(x)  = \frac{K-\rho}{x}+\frac{\rho(k-1)}{n-x}-\frac{\rho k}{n-\bar x}\geq\frac{P(x)}{x(n-x)(n-\bar x)},
	\end{equation*}
	where %$x_1 = (n-kx)/(n-k)$, and 
	$P$ is the polynomial of degree two  given by %and $K$ is replaced by $K_*=h(1/n)/\log(n)$. More explicitly :
	\begin{equation*}
		P(x) = (K_*-\rho)(n-x)\PAR{n-\frac{n-kx}{n-k}}+\rho(k-1)x\PAR{n-\frac{n-kx}{n-k}}-\frac{\rho k}{n-1}x(n-x).
	\end{equation*} By estimating the coefficients of $P$ we will show that it is positive for $x\in[0,n/k]$.
	Calculations show that the leading coefficient of $P$ is 
	\begin{equation*}
		\frac{k(k-1)}{n(n-k)}+\frac{k}{n(n-1)}-\frac{k}{n-k}(K_*-\rho)\geq \frac{k}{n-1}\PAR{\rho-(K_*-\rho)} > 0,
	\end{equation*}
	where we recall $K_* < 2\rho$. Furthermore, its coefficient of degree 1 is 
	\begin{equation*}
		k-1-\frac{k-1}{n-k}-(K_*-\rho)n\PAR{1-\frac{k+1}{n-k}}-\frac{k}{n-1}.
	\end{equation*}
	For $n\geq 5$, we have $(K_*-\rho)n = \frac{(n-1)\log(n/(n-1))}{\log(n)} < \frac{n-1}{n+1} < 1$, so that the coefficient of degree 1 is a concave function of $k$. Indeed,
	\begin{equation*}
		\frac{\partial^2}{\partial k^2}\PAR{-\frac{k-1}{n-k}+(K_*-\rho)n\frac{k+1}{n-k}} = 2\frac{(K_*-\rho)n(n+1)-(n-1)}{(n-k)^3} < 0
	\end{equation*} 
%	
%	For $n\geq 3$,
%	\begin{equation*}
%		\frac{n-1}{n\log(n)}\leq(K_*-\rho)n = \frac{(n-1)\log(n/(n-1))}{\log(n)} < 1,
%	\end{equation*}
%	so that the coefficient of degree 1 is a concave function of $k$ in the interval $[2,n-1]$ (\what {\em maybe add a bit of details on proof of concavity ?}). 
Therefore for $k\in [2,n-1]$ it is bounded from below by the minimum of its values at $k=2$ and $k=n-1$. (Note that if $n=3,4$ this is trivially true since either $k=2$ or $k=n-1$ in this case.) For $k=2$ we have
	\begin{equation*} 		
		1-\frac1{n-1}-(K_*-\rho)n\PAR{1-\frac{3}{n-1}}-\frac{2}{n-1}> 1-\frac1{n-1}-\PAR{1-\frac{3}{n-1}}-\frac{2}{n-1} =0,
	\end{equation*}
	and for $k=n-1$,
	\begin{equation*}
	\begin{split}
		n-2-(n-2)-(K_*-\rho)n\PAR{1-n}-1 &= (n-1)n(K_*-\rho) -1
		%\geq\frac{(n-1)^2}{n\log(n)}-1 
		> 0,
	\end{split}
	\end{equation*}
	where we use $K_*>1/(n-1)$. In conclusion, the first two coefficients of $P$ are positive.  Moreover,  the constant term is also positive because $P(0) = (n-n/(n-k))n(K-\rho) > 0$. This proves that for any $k\geq 2$, $P$ is positive on $\R_+$. Recalling that, for all $x\in [0,n/k]$,
	\begin{equation*}
		f'(x) \geq \frac{P(x)}{x(n-x)(n-x_1)}> 0,
	\end{equation*} 
	we see that $f$ is strictly increasing for $k\geq 2$. Since $f(1) = 0$, there are no solution of \eqref{eq:minf} for $k\geq 2$, in the domain $(1,n/k]$. 
	
	Summarizing, we have proved that any maximiser $a$ with $\mu(a)=1$ takes two values $x_1<x_2$, and that $x_2$ is taken only once, while $x_1$ is taken $n-1$ times. 
	%We now show that $x_1=0$ and $x_2=n$. 
	To conclude the prof, observe that if $k=1$, then $f$ is concave in $[0,n]$, and
	\begin{equation*}
		\begin{split}
			f(n) &= (K-K_*)\log(n) +(K_* -\rho)\log(n)-\rho(n-1)\log\PAR{\frac{n}{n-1}}\\
			& = (K-K_*)\log(n).
		\end{split}
	\end{equation*}
	Therefore, if $K>K_*$ there are no non-constant solutions of \eqref{eq:minf}. %which is impossible since the maximizer is a non-constant solution. 
	It follows that $K=K_*=h(1/n)/\log(n)$, and $x_1=0$, $x_2=n$. 
\end{proof}
\begin{remark}\label{rem:unique}
{\em 	The proof of Lemma \ref{lem:tec2} also shows that if $a\in\bbR_+^n$ is such that $\ent(a)\neq 0$ and $a$ has at least two non-zero entries then 
$$\widehat\psi_{1/n}(a)<\frac{h(1/n)}{\log n}\,\Ent(a),$$ so that the maximum in \eqref{eq:monf} is only achieved when $a$ is a multiple of a Dirac mass.}
\end{remark}
\begin{remark}\label{rem:eqcor}
{\em 	An equivalent  formulation of Lemma \ref{lem:tec2}  is
	\begin{equation}\label{eqcor}
		\sup_{a\in\R_+^n}\frac{\Ent(\bar a)}{\Ent(a)} = 1-\frac{\log(n-1)}{\log(n)},
	\end{equation}
	where the supremum is over all non-constant $a\in\R_+^n$, and 
the vector $\bar a$ is defined after \eqref{def:psihat}.
Indeed, if $\rho=1/n$, then
	\begin{equation*}
		\widehat\psi_\rho(a) = \frac{1}{n}\sum_{i=1}^n\rho a_i\log (a_i/\mu) +(1-\rho)\bar a_i\log(\bar a_i/\mu) = \rho\Ent(a)+(1-\rho)\Ent(\bar a).
	\end{equation*}
	So $\widehat\psi_\rho(a) \leq (h(1/n)/\log(n))\Ent(a)$ is equivalent to,
	\begin{equation*}
		\Ent(\bar a)\leq \frac{h(1/n)/\log(n)-\rho}{1-\rho}\Ent(a),
	\end{equation*}
	and \eqref{eqcor} follows from $\frac{h(1/n)/\log(n)-\rho}{1-\rho}=1-\frac{\log(n-1)}{\log(n)}$.
}\end{remark}

\subsection{Proof of Theorem \ref{th:meanfield1}}
  The upper bound 
  \begin{align}\label{pfth1}
\k[\a]\leq \sum_{\ell=2}^n w_\ell \,\frac{\binom{n-1}{\ell-1}\log \ell}{\log n},
\end{align}
  follows by taking a Dirac mass $f(x)=\ind_{x=x_0}$ at a fixed vertex $x_0\in V$. Indeed, here $\ent f = \frac1n \log n$ and for any $A\subset V$ with $|A|=\ell$ one has $\smallent_A f = \ind_{x_0\in A} \frac1\ell \log \ell$, and 
   \begin{align}\label{pfth2}
   \frac1n  \sum_{|A|=\ell}|A| \smallent_A f = \frac1{n} \binom{n-1}{\ell-1}  \log \ell,
\end{align}
  which implies \eqref{pfth1}.

To prove the lower bound, it is convenient to introduce the notation $\k_\ell(n)$ for the constant $\k[\a^\ell] $ when $|V|=n$.
We have
\begin{align}\label{pfth3}
 \k[\a] = \inf_f \frac{\sum_\ell w_\ell D_\ell(f)}{\ent f} \,,\qquad  D_\ell(f)= \frac\ell{n}  \sum_{|A|=\ell} \smallent_A f,
\end{align}
where the infimum is over all non-constant functions. Since $D_\ell(f)\geq \k_\ell(n) \ent f$, it is sufficient to prove that for any $2\leq \ell\leq n$ one has 
\begin{align}\label{pfth4}
 \k_\ell(n) 
\geq  \frac{\binom{n-1}{\ell-1}\log \ell}{\log n}.
\end{align}
We proceed by induction over $n$. Namely we assume that the above bound holds for $|V|=n-1$ and for all $2\leq \ell\leq |V|$ and show that this implies \eqref{pfth4}. For $\ell=|V|=2$ the statement is trivially true since $\k_2(2)=1$ in this case.    We use the notation $\eta_x\in\{0,1\}$ for the occupation variable at site $x$. Since we have one particle only, the configuration $\eta$ is everywhere zero except for a vertex where $\eta_x=1$, and the measure $\mu$ is uniform over all $\eta\in\{0,1\}^V$ such that $\sum_{x\in V}\eta_x=1$.  Then, for a fixed $x\in V$ we may decompose the entropy along the variable $\eta_x$:
\begin{align}\label{knr2}
\ent f =\mu\left[\ent(f|\eta_{x})\right]+
\ent \left[\mu(f|\eta_{x})\right].
\end{align} 
Here $\ent(f|\eta_{x}) = \mu\left(f\log(f/\mu(f|\eta_{x}))|\eta_{x}\right)$ is the entropy with respect to the conditional distribution $\mu(\cdot|\eta_{x})$. 
Thus, $\ent(f|\eta_{x}) = 0$ if $\eta_x=1$ and $$\mu\left[\ent(f|\eta_{x})\right] = \frac{n-1}{n}\ent(f|\eta_{x}=0).$$
Since $\k_n(n)=1$ for all $n$, the bound \eqref{pfth4} is trivially true for $\ell=n$ and we may assume $2\leq \ell\leq n-1$. By definition of $\k_\ell(n-1) $ one has
\begin{align}\label{knr3}
\ent(f|\eta_{x}=0)
&\leq \frac{\ell}{(n-1)\k_\ell(n-1)} \sumtwo{|A|=\ell:}{A\not\ni x}  \smallent_{A}f .
\end{align} 
Averaging over $x$ 
we find
\begin{align}\label{knr4}
\frac1n\sum_x \mu\left[\ent(f|\eta_{x})\right]
&\leq \frac{\ell(n-\ell)}{n^2\k_\ell(n-1)} \sum_{|A|=\ell}   \smallent_{A}f=\frac{n-\ell}{n\k_\ell(n-1)} D_\ell(f).
\end{align} 
We turn to the estimate of $\ent \left[\mu(f|\eta_{x})\right]$. 
Observe that 
\begin{align}\label{kn1l5}
\ent \left[\mu(f|\eta_{x})\right] = \psi_{1/n}(\mu(f|\eta_{x}=1),\mu(f|\eta_{x}=0)),
\end{align} 
where $\psi_\r$ was defined in \eqref{def:psir}, and 
note that $\mu(f|\eta_x=1)=f(x)$ while
\begin{align}\label{kn1l56}
\mu(f|\eta_x=0) &= \frac1{n-1} \sum_{y:\,y\neq x}f(y).
\end{align}
Therefore, by Lemma \ref{lem:tec2},
\begin{align}\label{kn1l58}
\frac1n\sum_{x\in V} \ent \left[\mu(f|\eta_{x})\right] &\leq 
\frac{h(1/n)}{\log n}\,\ent f.
\end{align} 
  Summarizing, we have proved that 
  \begin{align}\label{knr21}
\left(1-\frac{h(1/n)}{\log n}\right)\ent f\leq \frac{n-\ell}{n\k_\ell(n-1)} D_\ell(f)
.
\end{align} 
  Then, using $ 1-\frac{h(1/n)}{\log n} = \frac{(n-1)\log(n-1)}{n\log n}$, and assuming inductively the validity of \eqref{pfth4} for $\k_\ell(n-1)$, we have 
\begin{align}\label{kelln151}
\ent f
&
\leq \frac{(n-\ell)\log n}{(n-1)\k_\ell(n-1)\log(n-1)} \,D_\ell(f)  \leq
\frac{\log n}{\binom{n-1}{\ell-1}\log \ell}\,D_\ell(f) 
.
\end{align} 
This ends the proof of  \eqref{pfth4}.

It remains to prove the uniqueness of the minimizer. 
To this end it is sufficient to observe that if $f$ is such that $\ent f\neq 0$ and $f$ is not a multiple of a Dirac mass, then the the first inequality in  \eqref{kelln151} is strict for all $2\leq \ell\leq n-1$. 
This follows from the fact that \eqref{kn1l58} is a strict inequality in this case, see Remark \ref{rem:unique}. 

\subsection{Entropy constant of the star graph}
Consider the graph $G=S_n$ defined by $c_{xy}=\ind_{x_*\in xy}$, where $x_*\in V$ denotes the center of the star, and $|V|=n$. In contrast with the complete graph $K_n$, the entropy constant of the star is not achieved at a Dirac mass. However, the Dirac mass at a leaf $y\neq x_*$ gives a good approximation. 
\begin{proposition}\label{prop:star}
	The star graph $G=S_n$ satisfies, for  $n \geq 3$,
\begin{equation*}
		\frac{2\log(2)\PAR{1-\frac2n}}{\log(n-1)}\leq \kappa(S_n) \leq \frac{2\log (2)}{\log (n)},
	\end{equation*}
\end{proposition}

%\begin{remark}\label{rem:S_n}
%	{\em By definition $W(x\log(x)) = \log(x)$ and $W(n/2\log(n)) > W(n/2\log(n/2)) = \log(n)-\log(2)$ so the inequality \eqref{prop:star} always defines a non-empty interval.}
%\end{remark}
%\begin{corollary}
%	As $n\to\infty$ :
%	\begin{equation*}
%		\kappa(S_n) \sim\frac{2\log(2)}{\log(n)}.
%	\end{equation*}
%\end{corollary}
%\begin{proof}
%	The function $W$ satisfies
%	\begin{equation*}
%		W(x) = \log(x)-\log\log(x)+o(1).
%	\end{equation*}
%	Therefore,
%	\begin{equation*}
%		\begin{split}
%			\log(n)-W(n/2\log(n)) &= \log(2)-\log\log(n)+\log\log(n/2\log(n))+o(1)\\
%			&=\log(2)+o(1).
%		\end{split}
%	\end{equation*}\qedhere
%\end{proof}
\begin{proof}
%	We recall that the spectral gap satisfies $\lambda(S_n) = 1$ for all $n\geq 3$. 
Set 
	$$D_n(a) = \frac2n\sum_{x,y\in V}c_{xy} \smallent_{xy} (a) = \frac4n\sum_{y\neq x_*} \smallent_{x_*y} (a) \,,$$
	where $a\in\bbR_+^V$, $\smallent_{xy} (a)=\psi(a_x,a_y)$, and $\psi$ is defined in \eqref{def:psi}.  If $a$ is a Dirac mass on a leaf $y\neq x_*$,  then 
	\begin{equation*}
		\frac{D_n(a)}{\Ent(a)} = \frac{2\log(2)}{\log(n)}.
	\end{equation*} 
This gives the upper bound $\kappa(S_n) \leq 2\log2/\log(n)$. For the other direction it suffices to use the bound \eqref{kala3}, together with the well known fact that $\l(S_n)=1$ for all $n$.
\end{proof}
For $n=3$ we note that  
the estimate $\kappa(S_3)\leq\frac{2\log(2)}{\log (3)} $ is not useful 
since we already know that $\kappa(S_3)\leq \l(S_3)=1$.  In fact, one can show that $\k(S_3)=1$. We omit the details.

\subsection{Remarks on electric network reductions}
Let us recall that the electric network reduction at a node $x$ of a graph $G$ with vertex set $V$ and weights $c_{yz}$,  is the new graph $G_x$ with vertex set $V\setminus \{x\}$ characterized by the new weights, for $y,z\neq x$:
 \begin{align}\label{neweights}
\widetilde c^{\,\,x}_{yz} = c_{yz} + c^{*,x}_{yz}\,,\qquad c^{*,x}_{yz}  = \frac{c_{xy}c_{xz}}{\sum_{w\neq x} c_{xw}}\,.
\end{align} 
An important property satisfied by the spectral gap is the inequality $\l(G_x)\geq \l(G)$ for all weighted graphs $G$, and for all nodes $x$, see  \cite{CLR,Dieker} for a proof. 
It is not difficult to prove that the same monotonicity under reduction holds for the log-Sobolev constant $\b(G)$ and for the modified log-Sobolev constant $\varrho(G)$ defined in \eqref{kentbr}. 
\begin{lemma}\label{lem:redo}
For any weighted graph $G$, for any node $x$, 
\begin{align}\label{lsredo}
\b(G_x)\geq \b(G)\,,\qquad  \varrho(G_x)\geq \varrho(G).
 \end{align}
 \end{lemma}
\begin{proof}
We give the details of the proof of $\varrho(G_x)\geq \varrho(G)$ since the other inequality can be proved  in essentially the same way. Recall that 
\begin{align}\label{mls1}
 \varrho(G)=\inf_{f}\frac{\cE_G(f,\log f)}{\ent f},
 \end{align}
 where the infimum is over all non-constant $f:V\mapsto \bbR_+$, and we use the notation
 $$\cE_G(f,g)=-\mu((\cL_G f)g),
 $$
for all functions $f,g$,  with $\cL_G$ defined as in \eqref{L}.
Fix $x\in V$. If $f$ is  harmonic at $x$, that is $\cL_G f(x) =0$, then one checks that the electric network reduction $G_x$ of $G$ at $x$ satisfies 
$\cL_G f(y)=\cL_{G_x} f(y)$ at all $y\neq x$. Therefore
$$\cE_G(f,\log f)= \frac1n\sum_{y\neq x}(-\cL_G f)(y) \log f(y)  =\frac{n-1}{n}\,\cE_{G_x}(f,\log f)
.$$
Moreover for any $x$, setting %$\bar f = \frac1n\sum_y f(y)$ and 
$\bar f_x = \frac1{n-1}  \sum_{y\neq x} f(y)$, we have 
$$\mu(f) = \frac{n-1}n \bar f_x + \frac1n f(x),$$ 
and therefore by convexity of $t\mapsto t\log t$,
\begin{align}\label{entS1}
\ent f &=\frac1n\sum_{y\neq x} f(y)\log f(y) +\frac1n f(x)\log f(x)- \mu(f)\log \bar \mu(f)
\\& = \frac{n-1}n \ent_x(f) +  \frac{n-1}n \bar f_x\log  \bar f_x + \frac1n f(x)\log f(x)- \mu(f)\log \mu(f)\\& \geq \frac{n-1}n \,\ent_x(f) ,\end{align}
where $\ent_x(f)$ denotes the entropy of $f$ with respect to the uniform distribution on the reduced set of sites $V\setminus \{x\}$. Thus, restricting to $f$ harmonic at $x$ in \eqref{mls1}, but arbitrary outside of $x$, we have proved that  
\begin{align*}%\label{mls2}
 \varrho(G)\leq \inf_{f}\frac{\cE_{G_x}(f,\log f)}{\ent_x f}=\r(G_x).
\end{align*}\qedhere 
\end{proof}
\begin{remark}\label{rem:noredo}
{\em The entropy constant $\k(G)$ does not in general satisfy the above monotonicity property. For instance, taking  $G=S_n$, the star graph considered in Proposition \ref{prop:star}, a reduction at $x=x_*$ gives that $G'=G_x$ is the complete graph on $n-1$ vertices with $c_{yz}=1/(n-1)$. One can check numerically that $\k(S_4)  =0.9217860...$, while the reduction $G'$ of $S_4$ in the center has 
	\begin{equation*}
		\k(G')=\k(K_3)/3 = \frac{4\log(2)}{3\log(3)} = 0.8412396... < \k(S_4).
	\end{equation*}
On the other hand,  by the inequalities in Lemma \ref{lem:ele}, combined with Lemma \ref{lem:redo} one has always 
 \begin{align}\label{kls2}
 \k(G_x)\geq \log(2)\k(G).
 \end{align}
}\end{remark}
The next lemma shows that the monotonicity holds for the entropy constant as well whenever $x$ is a {\em leaf} of $G$, by which we mean that $x\in V$ is such that there exists a unique $y\in V$ with $c_{xy}>0$. 

\begin{lemma}\label{lem:tree}
Suppose that $x$ is a leaf of $G$. Then $\k(G_x)\geq \k(G)$. 
\end{lemma}
\begin{proof}
As in the proof of Lemma \ref{lem:redo} we have $\ent f\geq \frac{n-1}n\,\ent_x f$. Let $y\in V$ be the unique node with $c_{xy}>0$. Then, for any  $f:V\mapsto \bbR_+$ such that $f(y)=f(x)$, one has
$$
\frac2n\sum_{z,w\in V}c_{zw}\ent_{zw} f =\frac{n-1}{n} \,\frac2{n-1}\sum_{z,w\in V\setminus \{x\}}c_{zw}\ent_{zw} f.
$$
Since $f$ is arbitrary on $V\setminus \{x\}$ and the reduced graph $G_x$ is simply the graph $G$ without the edge $xy$ we conclude that 
$$
\k(G)\leq \frac{n-1}{n} \,\k(G_x) \frac{\ent_x f}{\ent f} \leq \k(G_x). 
$$
\end{proof}

\section{Independent particles with synchronous updates}
In this section we prove Theorem \ref{th:asyn} for the model of independent particles  with synchronous updates defined by the generator \eqref{geninda}.

\subsection{Proof of Theorem \ref{th:asyn}}
We prove only the statement for the entropy constant, since the proof works without modifications if we replace the entropy functionals by the variance functionals.
We start with the decomposition along the position of the first particle
\begin{align}\label{kgN1}
 \ent f = \nu\left[\ent(f|\xi_1)\right] + \ent[\nu(f|\xi_1)],
\end{align} 
where $\ent(f|\xi_1)=\nu\left[f\log (f/\nu(f|\xi_1))|\xi_1\right]$ is the entropy of $f$ w.r.t.\ the conditional distribution $\nu(\cdot|\xi_1)$.   
By definition of $\k[\a,N]$ we have
\begin{align}\label{kgN2}
\k[\a,N-1]\,\ent(f|\xi_1)\leq  \sum_{A\subset V}\a_A\,\nu\left[\ent_{A}(f|\xi_1)|\xi_1\right],
\end{align} 
where $\ent_{A}(f|\xi_1)=\ent(f|\xi_1,\{\eta_z,z\notin A\})$ is the conditional entropy given all occupation variables outside $A$ and given the position of the particle labeled $1$. 
Integrating,
\begin{align}\label{kgN3}
\k[\a,N-1]\,\nu\left[\ent(f|\xi_1)\right] \leq  \sum_{A\subset V}\a_A\,\nu\left[\ent_{A}(f|\xi_1)\right].
\end{align}
Another application of the decomposition, this time for $\ent_{A}$, shows that 
\begin{align}\label{kgN4}
\nu\left[\ent_{A}(f|\xi_1)\right] = \nu\left[\ent_{A}(f)\right] -  \nu\left[\ent_{A}(\nu_{A}(f|\xi_1))\right].
\end{align}
On the other hand, since $\nu(f|\xi_1)$ is a function of one particle, the 1-particle entropic constant $\k[\a]$ satisfies 
\begin{align}\label{kgN5}
\k[\a]\, \ent[\nu(f|\xi_1)]\leq  \sum_{A\subset V}\a_A\,\nu\left[\ent_{A}(\nu(f|\xi_1))\right],
\end{align} 
We are going to prove that for all $A\subset V$, $|A|\geq 2$, 
\begin{align}\label{kgconv}
\nu\left[\ent_{A}(\nu(f|\xi_1))\right]\leq \nu\left[\ent_{A}(\nu_{A}(f|\xi_1))\right].
\end{align} 
If \eqref{kgconv} holds, then by \eqref{kgN1}-\eqref{kgN5} we have obtained 
\begin{align}\label{kgN6}
 \min\{\k[\a,N-1],\k[\a]\}\,\ent f\leq \sum_{A\subset V}\a_{A}\,\nu\left[\ent_{A}(f)\right],
\end{align} 
that is $\k[\a,N]\geq \min\{\k[\a,N-1],\k[\a]\}$ for all $N\ge 2$. Iterating, this proves  $\k[\a,N]=\k[\a]$ for all $N\geq 1$.

To prove \eqref{kgconv},  we write 
$$\nu_{A}(f|\xi_1)=\nu(f|\xi_1,\{\eta_z,z\notin A\}) = \prod_{i\in\eta_A, i\neq 1}[\mu_{i,A} f](\xi),$$
where $\mu_{i,A}$ denotes the stochastic operator that equilibrates uniformly in $A$ the $i$-th coordinate $\xi_i$. Note that 
the operators $\mu_{i,A}, \mu_{j,B}$ commute for  $i\neq j$ and any $A,B\subset V$. With this notation $\nu=\prod_{i=1}^N \mu_{i,V}$ and 
$$\nu(f|\xi_1)=\prod_{ i\neq 1}\mu_{i,V}f =\prod_{ i\neq 1}\mu_{i,V}\nu_{A}(f|\xi_1) .$$  
On the other hand,
$$
\nu\left[\ent_{A}(\nu(f|\xi_1))\right]=\nu\left[{\bf 1}(\xi_1\in A)\,\smallent_A \varphi\right]= \frac{|A|}n\,\smallent_A \varphi,
$$
where $\varphi (x) =  \nu(f|\xi_1=x)$, $x\in V$. By the convexity of $\varphi \mapsto \smallent_A \varphi$, we have 
$$
\smallent_A \varphi\leq \prod_{ i\neq 1}\mu_{i,V}\,\smallent_A \nu_{A}(f|\xi_1),
$$
where $\smallent_A \nu_{A}(f|\xi_1)$ is defined as $\smallent_A \varphi_A$ with $\varphi_A(y) = \nu_{A}(f|\xi_1=y)$, for all $y\in V$. Note that $\smallent_A \nu_{A}(f|\xi_1)$ depends only on the particle positions $\xi_j$, $j\notin \eta_A$. 
%Similarly,
%$$
%\mu\left[\ent_{xy}(\mu_{xy}(f|\xi_1))\right] =\mu\left[{\bf 1}(\xi_1\in xy) \psi\left(\mu_{xy}(f|\xi_1),\mu_{xy}(f^{1,xy}|\xi_1)
%%\prod_{i\in\eta_x\cup\eta_y, i\neq 1}p_{i,xy}f,\prod_{i\in\eta_x\cup\eta_y, i\neq 1}p_{i,xy}f^{1,xy}
%\right)\right]
%$$
Therefore, %by the joint convexity of $\psi$,
\begin{align}\label{kgN8}
\nu\left[\ent_{A}(\nu(f|\xi_1))\right]&\leq\nu\Big[{\bf 1}(\xi_1\in A)
\prod_{ i\neq 1}
\mu_{i,V}\,\smallent_A \nu_{A}(f|\xi_1)
%\prod_{i\in\eta_x\cup\eta_y, i\neq 1}p_{i,xy}\psi(f,f^{1,xy})\right]\\&=\mu\left[{\bf 1}(\xi_1\in xy)\psi(f,f^{1,xy})
\Big]\\& = 
\nu\Big[
\prod_{ i\neq 1,i\notin \eta_A}\mu_{i,V}{\bf 1}(\xi_1\in A)
\,\smallent_A \nu_{A}(f|\xi_1)
%\prod_{i\in\eta_x\cup\eta_y, i\neq 1}p_{i,xy}\psi(f,f^{1,xy})\right]\\&=\mu\left[{\bf 1}(\xi_1\in xy)\psi(f,f^{1,xy})
\Big]=
\nu\left[\ent_{A}(\nu_{A}(f|\xi_1))\right] .
%\leq\sum_{i=1}^N \mu\left[{\bf 1}(\xi_i\in xy)\psi(f,f^{i,xy})\right]\\&=\mu\left[\ent_{xy}(\mu_{xy}(f|\xi_1))\right].
\end{align} 
This ends the proof.

\begin{remark}\label{rem:inhom}
{\em The same proof with essentially no modification can be used to show a more general result where the uniform  measure $\nu=\mu^N$ is replaced by a non-uniform product measure $\nu=\prod_{i}\tilde \mu_i$, for some arbitrary distributions $\tilde \mu_i$ over $V$, and the process is formally defined again as in \eqref{geninda}. 
} 
\end{remark}

\section{Permutations}
In this section we prove our results about the mean field spectral gap, Theorem \ref{th:meanfieldgap}, and about the mean field entropy constant, Theorem \ref{th:meanfield2}. Then we prove Corollary \ref{cor:permanent}. Next, we compute the entropy constant of the Bernoulli-Laplace model, Theorem \ref{th:BL}. Finally, we give an example of network reduction, in the case of the star graph. Throughout this section $\mu$ denotes the uniform distribution over the symmetric group $\cS_n$. As usual $V$ is the vertex set, with $|V|=n$.

\subsection{Proof of Theorem \ref{th:meanfieldgap}}
To prove the upper bound on $\l[\a,\cS_n]$ we consider a function of the single particle position $\xi_1$, namely $f(\si)=\varphi(\xi_1)$ for some $\varphi:V\mapsto\bbR_+$. In this case, for any $A\subset V$ with $|A|=\ell$,
$$
\mu\left[\var_A f\right] = \mu\left[\ind_{\xi_1\in A}\smallvar_A \varphi \right] =\frac{\ell}n \,\smallvar_A \varphi =\frac{1}{\ell n}\sum_{x,y\in A}(\varphi(x)-\varphi(y))^2 .
$$
It follows that 
$$%\var f = \frac1n \sum_x \varphi(x)^2,\qquad 
\sum_{|A|=\ell}\mu\left[\var_A f\right]  
= \frac{n}{\ell}\binom{n-2}{\ell-2}\var f.$$
Therefore, any mean zero function of a single label is an eigenfunction for the mean field $\a$-shuffle operator $\cG_\a$ with eigenvalue given by \eqref{meanfieldgap1}. 
We have to prove that any other eigenfunction has a strictly larger eigenvalue. It is convenient to use the notation $\L_\ell(n)=\l[\a^\ell]$.
Let us first prove %the desired statement in the case $\a=\a^\ell$,  that is let us show 
that for all $n\geq 2$, $\ell=2,\dots,n$, 
% \begin{align}\label{univar1}
%\var f \leq \frac{\ell}{n\binom{n-2}{\ell-2}}\sum_{|A|=\ell}\mu\left[\var_A f\right].
%\end{align}
%Thus we need to prove
 \begin{align}\label{univar2}
\L_\ell(n)=\frac{n}{\ell}\binom{n-2}{\ell-2}.
\end{align}
We prove \eqref{univar2} by induction over $n$. The case $n=2$ is trivial. More generally, $\ell=n\geq 2$ is also trivial, so we pick  $\ell\in\{2,\dots,n-1\}$. We decompose the variance along the random label $\si_x$: 
 \begin{align}\label{univar3}
\var f = \mu\left[\var(f|\si_x)\right]+
\var \left[\mu(f|\si_x)\right] .
\end{align} 
By definition of $\L_\ell(n)$, 
\begin{align}\label{univar4}
\var(f|\si_x)
\leq \frac1{\L_\ell(n-1)}\sum_{|A|=\ell}\mu\left[\var_{A}f|\si_x\right]\IND_{x\notin A}.
\end{align} 
Averaging over $x\in V$, and taking the expectation with respect to $\mu$, 
\begin{align}\label{univar5}
\frac1n\sum_x\mu\left[\var(f|\si_x)\right]\leq \frac{n-\ell}{n\L_\ell(n-1)}\sum_{|A|=\ell}\mu\left[\var_{A}f\right].
\end{align} 
Using the bound in \cite[Section 5]{Shonan}
 one has the estimate, for all $f:\cS_n\mapsto\bbR$,
\begin{align}\label{univar6}
\sum_{x\in [n]}\var \, \mu(f|\si_x)
\leq \frac{n}{n-1}\, \var f. 
\end{align}
It follows from \eqref{univar3}-\eqref{univar6} that 
 \begin{align}\label{univar7}
\var f \leq \frac{(n-1)(n-\ell)}{(n-2)n\L_\ell(n-1)}\sum_{|A|=\ell}\mu\left[\var_{A}f\right].
\end{align} 
Assuming inductively that $\L_\ell(n-1)=\frac1{\ell}(n-1)\binom{n-3}{\ell-2}$, \eqref{univar7} shows that 
$$
\L_\ell(n)\geq \frac{n}{\ell}\binom{n-2}{\ell-2}.
$$
%Since this is the value obtained taking a mean zero function of a single label, t
This completes the proof of \eqref{univar2} for all $n\geq 2$, $\ell=2,\dots,n$. To prove \eqref{meanfieldgap1} it suffices to observe that if $\a=\sum_\ell w_{\ell=2}^n \a^\ell$ then  necessarily
\begin{align}\label{univarco}
\l[\a,\cS_n] \geq \sum_{\ell=2}^nw_\ell\, \L_\ell(n).
\end{align}
 To prove the uniqueness part, we note that if $f$ is an eigenfunction of $\cG_\a$ that is orthogonal to all eigenfunctions of a single particle, then $\mu(f|\si_x)=0$ for all $x$, see e.g.\ \cite[Section 2.3]{CLR}. In particular, the left hand side in \eqref{univar6} must vanish, so that \eqref{univar7} becomes a strict inequality. This ends the proof of the theorem. 
  
\subsection{Proof of Theorem \ref{th:meanfield2}}
If $f$ is a Dirac mass $f=\IND_{\si_0}$ for any given permutation $\si_0\in\cS_n$, then for any $A\subset [n]$ with $|A|=\ell$,
$$\mu\left[\ent_{A}f\right]=\frac{\log \ell!}{n!}\,,\qquad \ent f =  \frac1{n!}\log n!,$$
and therefore one has the upper bound 
\begin{align}\label{kellr1}
\k[\a,\cS_n]\leq \sum_{\ell=2}^n w_\ell \frac{\binom{n}\ell \log \ell !}{\log n!},
\end{align} 
for all mean field cases. 

To prove the lower bound, it is convenient to introduce the notation $K_\ell(n) = \k[\a^\ell,\cS_n]$. Arguing as in \eqref{univarco}, it is then sufficient to prove
 that for all $n\geq 2$, for all $\ell=2,\dots,n$, 
 \begin{align}\label{keller1}
K_\ell(n)\geq  \frac{\binom{n}\ell \log \ell !}{\log n!}
\end{align} 
We prove this by induction over $n$. Clearly, the case $\ell=n$ is trivial. Thus, we pick $\ell\in\{2,\dots,n-1\}$. 
We start with the decomposition along the variable $\si_x$:
\begin{align}\label{kellr3}
\ent f = \mu\left[\ent(f|\si_x)\right]+
\ent \left[\mu(f|\si_x)\right] .
\end{align} 
By definition of $K_\ell(n)$,
\begin{align}\label{kellr4}
\ent(f|\si_x)\leq \frac{1}{K_\ell(n-1)}\sum_{|A|=\ell}\mu\left[\ent_{A}f|\si_x\right]\IND_{x\notin A}.
\end{align} 
Averaging over $x\in[n]$, and taking the expectation with respect to $\mu$,
\begin{align}\label{kellr5}
\frac1n\sum_x\mu\left[\ent(f|\si_x)\right]\leq \frac{n-\ell}{nK_\ell(n-1)}\sum_{|A|=\ell}\mu\left[\ent_{A}f\right].
\end{align} 
We turn to an estimate on the last term in \eqref{kellr3}. We start with the following useful observation.

\begin{lemma}\label{lem:equiv}
The statement \eqref{keller1} at $\ell=n-1$, that is $K_{n-1}(n) \geq n\log((n-1)!)/\log(n!)$, is equivalent to
the statement \begin{align}\label{kellero4}
\sum_x\ent \,\mu(f|\si_x)
\leq \,\frac{n\log n}{\log(n!)}\,\,\ent f,
\end{align} 
for all $f:\cS_n\mapsto\bbR_+$.
\end{lemma}
\begin{proof}
By definition of the entropy constant $K_{n-1}(n) $ one has
\begin{align}\label{kellero51}
K_{n-1}(n)\,\ent f\leq \sum_{x\in V}\mu\left[\ent_{V\setminus \{x\}} f\right],
\end{align}
for all $f:\cS_n\mapsto \bbR_+$. 
Using \eqref{kellr3}, and the identity $\ent(f|\si_x)=\ent_{V\setminus \{x\}} f$, we see that if $ K_{n-1}(n) \geq n\log((n-1)!)/\log(n!)$ then
\begin{align}\label{kellero5}
\frac1n\sum_x\ent \,\mu(f|\si_x) &=  \ent f -  \frac1n\sum_x\mu\left[\ent_{V\setminus \{x\}} f\right]\\& \leq 
(1- K_{n-1}(n)/n)\,\ent f =\frac{\log n}{\log(n!)}\,\ent f. 
\end{align}
The other implication is also an immediate consequence of the first identity in \eqref{kellero5}.
\end{proof}
Let us now show how to use Lemma \ref{lem:equiv} to conclude the proof of Theorem \ref{th:meanfield2}. As in the proof of Theorem \ref{th:meanfield1}, let $\k_\ell(n)$ denote the 1-particle entropy constant $\k[\a^\ell]$. We have shown that 
\begin{align}\label{pfth400}
 \k_\ell(n) 
=  \frac{\binom{n-1}{\ell-1}\log \ell}{\log n}.
\end{align}
Setting $\varphi_x(i)=\mu(f|\si_x=i)$, $i=1,\dots,n$, by definition of $ \k_\ell(n) $ one has
\begin{align}\label{kellr10}
\ent \left[\mu(f|\si_x)\right]\leq \frac{\ell}{n\k_\ell(n)}\, \sum_{|B|=\ell}{\rm ent}_B \varphi_x .
\end{align} 
We view $B$ as a set of labels, and write $\cS_B$ for the set of permutations of the labels in $B$. We introduce the notation, for $u=\{u_\pi,\,\pi\in \cS_B\}\in\bbR_+^{\ell!}$,
\begin{align}\label{kellr91}
\Psi(u) = \frac1{\ell!}\sum_{\pi\in \cS_B} u_\pi\log (u_\pi/\bar u)\,,\qquad \bar u =  \frac1{\ell!}\sum_{\pi\in \cS_B} u_\pi.
\end{align}
We note that if $i\in B$ is fixed and $u_\pi = \mu(f|\si_x=\pi(i))$, with $\pi\in \cS_B$, then %for all functions $f$,
\begin{align}\label{kellr92}
{\rm ent}_B \varphi_x=\Psi(\{\mu(f^\pi|\si_x=i),\,\pi\in \cS_B\}),
% = \frac1{\ell!}\sum_{\pi\in \cS_B} u_\pi\log (u_\pi/\bar u)\,,\qquad \bar u =  \frac1{\ell!}\sum_{\pi\in \cS_B} u_\pi.
\end{align}
where we use the relation 
$$
\mu(f|\si_x=\pi(i))= \mu(f^\pi|\si_x=i),
$$
with the notation $f^\pi$ for the function $f$ calculated at a configuration after the labels in $B$ have been rearranged according to $\pi$. 
Next, notice that 
$$
\mu(f|\si_x=\pi(i))=\mu(g_{i,B}|\si_x=\pi(i))= \mu(g_{i,B}^\pi|\si_x=i),
$$
where $g_{i,B} := \mu(f|\xi_i,\xi_{B^c})$ denotes the conditional expectation of $f$ given the positions of label $i$ 
and the positions of all labels $j\notin B$. 
Since $u\mapsto \Psi(u)$ is  convex, one has, for any $i\in B$:
\begin{align}\label{kellr93}
{\rm ent}_B \varphi_x &= \Psi(\{\mu(g_{i,B}^\pi|\si_x=i),\,\pi\in \cS_B\})\\& 
\leq \mu\left(\Psi(\{g_{i,B}^\pi,\,\pi\in \cS_B\})| \si_x=i\right)
= n\,\mu\left(\Psi(\{g_{i,B}^\pi,\,\pi\in \cS_B\}); \si_x=i\right).
\end{align}
This holds for all $i\in B$ fixed, and noting that $\Psi(\{g_{i,B}^\pi,\,\pi\in \cS_B\})$ does not depend on the choice of the label $i\in B$, one has for any $i\in B$:
\begin{align}\label{kellr94}
{\rm ent}_B \varphi_x &
\leq \frac{n}\ell\,\mu\left(\Psi(\{g_{i,B}^\pi,\,\pi\in \cS_B\}); \si_x\in B\right).
\end{align}
Averaging over $x\in[n]$ we obtain, for any $i\in B$:
 \begin{align}\label{kellr95}
\frac1n\sum_x{\rm ent}_B \varphi_x &
\leq \mu\left(\Psi(\{g_{i,B}^\pi,\,\pi\in \cS_B\})\right).
\end{align}
From \eqref{kellr10} it follows that
\begin{align}\label{kellr96}
\frac1n\sum_x
\ent \left[\mu(f|\si_x)\right]&\leq 
\frac{\ell}{n\k_\ell(n)}\, \sum_{|B|=\ell}\mu\left(\Psi(\{g_{i,B}^\pi,\,\pi\in \cS_B\})\right)\\&=
\frac{1}{n\k_\ell(n)}\, \sum_{|B|=\ell}\sum_{i\in B}\mu\left(\Psi(\{g_{i,B}^\pi,\,\pi\in \cS_B\})\right).
\end{align}
There is precisely one set of vertices $A\subset V$ such that $\si_A=\xi_B$ and for such $A$ there is precisely one $x\in A$ such that $\si_x=i$, and  in this case $$\ent_A\mu_A(f|\si_x) =\ent_A \mu(f|\si_x,\si_{A^c}) = \Psi(\{g_{i,B}^\pi,\,\pi\in \cS_B\}).$$ 
Therefore, we arrive at 
\begin{align}\label{kellr961}
\frac1n\sum_x
\ent \left[\mu(f|\si_x)\right]&\leq 
\frac{1}{n\k_\ell(n)}\, \sum_{|A|=\ell}\sum_{x\in A}
\mu\left(\ent_A \mu_A(f|\si_x) \right).
\end{align}
%\bigskip
%
%We view $B$ as a set of labels, and for any set $A\subset [n]$ with $|A|=|B|$, and any $\si\in S_n$,  we write $\si_A\sim B$ if the (unordered) set of labels in $\si_A$ coincides with the set $B$. For any $|A|=|B|=\ell$, $x\in A$, and any $\si$ such that $\si_A\sim B$ we have
%$$
%{\rm ent}_B \varphi_x = \ent_A  \mu(f|\si_x).
%$$
%Thus, we may write, for all $B$ and $x$:
%$$
%{\rm ent}_B \varphi_x = \frac1{\binom{n-1}{\ell-1}}\sum_{A \ni x} \mu\left[\ent_A  \mu(f|\si_x)|\si_A\sim B\right],
%$$
%where it is understood that the sum is restricted to  $|A|=\ell$. Since if $x\notin A$ then $\ent_A \mu(f|\si_x)=0$, we may drop the restriction $x\in A$ in the sum. Since $\mu(\si_A\sim B) = 1/\binom{n}{\ell}$, we have 
%$$
%{\rm ent}_B \varphi_x = \frac{n}\ell\sum_{A} \mu\left[\ent_A  \mu(f|\si_x) \IND_{\si_A\sim B}\right].
%$$
%Finally, given $\si$ and $A$, there is only one $B\subset[n]$ such that $\si_A\sim B$. Therefore
% \begin{align}\label{kellr11}
%\ent \left[\mu(f|\si_x)\right]\leq \frac{2}{\k_\ell(n,1)}\, \sum_{|A|=\ell}\mu\left[\ent_A  \mu(f|\si_x)\right] .
%\end{align}
Since we assume inductively that \eqref{keller1} holds up to $n-1$ and for all $\ell=2,\dots,n-1$, it follows from Lemma \ref{lem:equiv} that for any $|A|<n$ we have the estimate
\begin{align}\label{keller96}
\sum_{x\in A}
\ent_A \mu_A(f|\si_x) \leq \frac{|A|\log (|A|)}{\log (|A| !)}\,\ent_A f.
\end{align}
Therefore, for any $\ell=2,\dots,n-1$, 
\begin{align}\label{kellr960}
\frac1n\sum_x
\ent \left[\mu(f|\si_x)\right]&\leq 
\frac{\ell\log \ell}{n\k_\ell(n)\log(\ell!)}\, \sum_{|A|=\ell}
\mu\left[\ent_A f \right].
\end{align}
Using \eqref{pfth400}, we see that $$
\frac{\ell\log \ell}{n\k_\ell(n)\log(\ell!)} = \frac{\log n}{\binom{n}{\ell}\log \ell !} .
$$
In conclusion, combining \eqref{kellr960} and \eqref{kellr5} we have shown that 
\begin{align}\label{kellera5}
\ent f
\leq \left(\frac{(n-\ell)}{nK_\ell(n-1)} + \frac{\log n}{\binom{n}{\ell}\log \ell !}\right) \sum_{|A|=\ell}
\mu\left[\ent_A f \right].
\end{align} 
Using the inductive assumption \eqref{keller1} again for $K_\ell(n-1)$ we obtain
\begin{align}\label{kelleri5}
\ent f
\leq \frac{\log(n!)}{\binom{n}\ell \log(\ell!)} \sum_{|A|=\ell}
\mu\left[\ent_A f \right],
\end{align} 
which proves the inequality \eqref{keller1}. This proves \eqref{th:meanfield2}. 

To prove the uniqueness of the minimizers, observe that by the uniqueness in Theorem \ref{th:meanfield1} the only way to saturate the inequality \eqref{kellr10} for all $x$ is to have 
$\mu(f|\si_x)$ a multiple of a Dirac mass for all $x$, say $\mu(f|\si_x)=t_x\d_{i_x}$ for some $t_x>0$ and some label $i_x$. However, $\mu(f) = \mu(\mu(f|\si_x))$ implies $t_x=t=n\mu(f)$ for all $x$, for some $t>0$. Then $g(\si)/n!=n f(\si)/(tn!)$ is a probability on $\cS_n$ with marginal $\d_{i_x}(\si_x)$ at $\si_x$.  
Since all marginals are deterministic it follows that $g$ is deterministic and thus $g(\si)=n!\d_{\si_0}$ for some $\si_0\in\cS_n$, that is $f$ is a multiple of a Dirac mass. This ends the proof of Theorem \ref{th:meanfield2}.

\subsection{Proof of Corollary \ref{cor:permanent}}
Set $p_c= \frac{n\log n}{\log(n!)}$. %We are going to prove 
From Theorem \ref{th:meanfield2} and Lemma \ref{lem:equiv} we know that for any $f:\cS_n\mapsto\bbR_+$ one has
\begin{align}\label{kelleros4}
\sum_x\ent \,\mu(f|\si_x)
\leq p_c\,\ent f.
\end{align} 
For any collection of functions $\varphi_x: V\mapsto \bbR_+$, an application of \eqref{kelleros4} with the choice $f(\si)=\prod_x \varphi_x(\si_x)$ and the variational principle for entropy show that 
\begin{align}\label{perma3}
\mu\left[\prod_{x\in V}\varphi_x(\si_x)\right]\leq \prod_{x\in V}\mu\left[\varphi_x(\si_x)^{p_c}\right]^{1/p_c},
 \end{align} 
 see \cite[Theorem 2.1]{CarlenCordero}. 
Consider now the matrix $A=(a_{x,y})$ such that $a_{x,y}=\varphi_x(y)$. The left hand side above equals $(1/n!){\rm perm}(A)$, while for every $x$:
$$
\mu\left[\varphi_x(\si_x)^{p_c}\right]^{1/p_c} = n^{-1/p_c}\|R_x\|_{p_c} = \left(\frac1{n!}\right)^{1/n}\|R_x\|_{p_c},
$$ 
where $R_x$ denotes the $x$-th row of $A$. Therefore \eqref{perma3} proves the statement \eqref{perma2}. Moreover, the argument leading from \eqref{kelleros4} to \eqref{perma3} also shows that  if \eqref{perma3} is an equality for some functions $\varphi_x$, then 
\eqref{kelleros4} must be an equality with $f(\si)=\prod_x \varphi_x(\si_x)$, see \cite[Theorem 2.2]{CarlenCordero}. 
By the uniqueness in Theorem \ref{th:meanfieldent} this is only possible if either $f$ is constant, or if $f$ is a multiple of a Dirac mass. 
  In the first case $A$ is a scalar multiple of the all $1$ matrix, while in the second it is the identity matrix up to multiplication by a scalar and up to permutation of the rows. This proves the corollary at $p=p_c$. As already noted in \cite[Lemma 1]{Samorodnitsky}, this is sufficient to prove the desired statement for all $p\geq 1$.

%\what {\em note that we did not need to prove the bound \eqref{kelleros4} directly, it came from the induction in the proof of Theorem \ref{th:meanfield2} !! }\what 

\subsection{Proof of Theorem \ref{th:BL}}
Here we consider the entropy constant for Bernoulli-Laplace and prove the identity stated in Theorem \ref{th:BL}.
We have $r$ indistinguishable particles and $\mu$ is the uniform distribution over all $\binom{n}r$ configurations. 
We may rewrite the entropy constant $\k(n,r)$ as the best constant $\k\geq 0$ in the inequality
\begin{align}\label{knr}
\k(n,r)\,\ent f  
\leq \frac12\sum_{x,y} \mu(\ent_{xy}f).
\end{align} 
%Thus, $\k(n,r)$ is obtained as in the definition of $\k[\a^2,\cS_n]$ but we restrict to functions $f$ of $r$ particles only and such that $f$ is symmetric with respect to the exchange of position of the particles.  We are going to compute $\k(n,r)$ for all $n,r$. In particular, when $r=1$, this gives an alternative computation of $\k(K_n)$. 
 We use the notation $\eta_x\in\{0,1\}$ for the occupation variable at site $x$, and write $\r=r/n$ for the particle density. The upper bound $$\k(n,r)\leq \frac{r(n-r)\log(2)}{\log\binom{n}r}$$
 follows by choosing test function $f=\IND(\eta=\eta_0)$ for some fixed configuration $\eta$. Indeed, in this case one has $$\ent f = \log\binom{n}r/\binom{n}r\,,\quad \frac12\sum_{y,z}\mu(\ent_{yz}f) = r(n-r)\log(2)/\binom{n}r.$$
For the lower bound we use the recursive approach as in the proof of Theorem \ref{th:meanfield1}. Thus, for a fixed $x$ we write
\begin{align}\label{knr2bl}
\ent f =\mu\left[\ent(f|\eta_{x})\right]+
\ent \left[\mu(f|\eta_{x})\right] ,
\end{align} 
and observe that by definition of $\k(n,r) $ one has
\begin{align}\label{knr3bl}
&\mu\left[\ent(f|\eta_{x})\right]
=\r\,\ent(f|\eta_{x}=1)+(1-\r)\ent(f|\eta_{x}=0)\\&
\qquad\leq  \frac{\r}{2\k(n-1,r-1)} \sum_{y,z\neq x} \mu(\ent_{yz}f|\eta_x=1) + 
\frac{1-\r}{2\k(n-1,r)} \sum_{y,z\neq x} \mu(\ent_{yz}f|\eta_x=0).
\end{align} 
Note that 
$$
\ent_{yz}f=\psi(f,f^{yz})\eta_y(1-\eta_z) + \psi(f,f^{yz})\eta_z(1-\eta_y),
$$
where $\psi$ is defined in \eqref{def:psir} and $f^{yz}$ denotes the function $f$ evaluated at the configuration where the occupation variables at $y$ and $z$ have been swapped. Averaging over $x$, 
\begin{align}\label{knr4bl}
\frac1n\sum_x \mu\left[\ent(f|\eta_{x})\right]
&\leq \left(\frac{r-1}{n\k(n-1,r-1)}+\frac{n-r-1}{n\k(n-1,r)}\right) \frac12\sum_{y,z} \mu(\ent_{yz}f).
\end{align} 
We turn to the estimate of $\ent \left[\mu(f|\eta_{x})\right]$. Recalling the definition \eqref{def:psir}, %we have
\begin{align}\label{knr5}
\ent \left[\mu(f|\eta_{x})\right] = \psi_\r(\mu(f|\eta_{x}=1),\mu(f|\eta_{x}=0).
\end{align} 
We write
\begin{align}\label{knr56}
\mu(f|\eta_x=0) &= \frac1{n\r(1-\r)} \sum_{y} \mu(f(1-\eta_x)\eta_y)
% = \frac1{(n-1)!} \sum_{\si} \sum_y f(\si)\IND(\si_x=j)\IND(\si_y=i)\\& =
\\ & = \frac1{n\r(1-\r)} \sum_{y} \mu(f^{xy}(1-\eta_y)\eta_x)
\\&= \frac1{n(1-\r)} \sum_{y} \mu(f^{xy}(1-\eta_y)|\eta_x=1).
\end{align}
By convexity of $\psi_\r$,
\begin{align}\label{ipkn7}
\psi_\r(\mu(f|\eta_{x}=1),\mu(f|\eta_{x}=0)& \leq  \frac1{n(1-\r)} \sum_{y} \mu((1-\eta_y)\psi_\r(f,f^{xy})|\eta_x=1)
\\& = \frac1{n\r(1-\r)} \sum_{y} \mu(\eta_x(1-\eta_y)\psi_\r(f,f^{xy})) 
.
\end{align}
Summing over all $x$ we have obtained 
\begin{align}\label{knr9}
\frac1n\sum_x\ent \left[\mu(f|\eta_{x})\right] &\leq  \frac{1}{n^2\r(1-\r)} \sum_{x,y} \mu(\eta_x(1-\eta_y)\psi_\r(f,f^{xy})).
\end{align} 
Since $f^{xy}=f^{yx}$, % we have:
\begin{align}\label{knr009}
\sum_{x,y} \mu(\eta_x(1-\eta_y)\psi_\r(f,f^{xy}))=\sum_{x,y} \mu(\eta_y(1-\eta_x)\psi_\r(f,f^{xy})).
\end{align} 
Since $\mu(g)=\mu(g^{xy})$ for any function $g$, we also have, for all $x,y$: $$\mu(\eta_y(1-\eta_x)\psi_\r(f,f^{xy}))=\mu(\eta_x(1-\eta_y)\psi_\r(f^{xy},f)).$$
If we define $\bar\psi_\r(s,t)=\frac12(\psi_\r(s,t)+\psi_\r(t,s))$, we have shown that
\begin{align}\label{knr0009}
\frac1n\sum_x\ent \left[\mu(f|\eta_{x})\right] &\leq  \frac{1}{n^2\r(1-\r)} \sum_{x,y} \mu(\eta_x(1-\eta_y)\bar\psi_\r(f,f^{xy}))\\&=
\frac{1}{2n^2\r(1-\r)} \sum_{x,y} \mu(\bar\psi_\r(f,f^{xy})).
\end{align} 
From Corollary \ref{cor:tec1} we see that 
\begin{align}\label{ipkn08}
\bar \psi_\r(s,t) \leq  \frac{h(\r)}{\log(2)}\,\psi(s,t),
\end{align}
for all $\r\in[0, 1]$, $s,t>0$, with $h(\r)=-\r\log \r - (1-\r)\log(1-\r)$.

From \eqref{knr4bl} and \eqref{knr0009},
\begin{align}\label{knr10}
\frac{1}{\k(n,r)}\leq \frac{r-1}{n\k(n-1,r-1)}+\frac{n-r-1}{n\k(n-1,r)} + \frac{h(\r)}{n^2\r(1-\r)\log(2)}.
\end{align}
Assume inductively that $\k(n-1,r-1)\geq \bar\k(n-1,r-1)$ and $\k(n-1,r)\geq \bar\k(n-1,r)$, where 
\begin{align}\label{thknr0bar}
\bar\k(n,r)=\frac{r(n-r)\log(2)}{\log\binom{n}r}.
\end{align}  
Using the relations $\binom{n}r = \frac{n}r\binom{n-1}{r-1}=\frac{n}{n-r}\binom{n-1}{r}$ one has
\begin{align}\label{knr100}
\frac{r-1}{n\bar \k(n-1,r-1)}+\frac{n-r-1}{n\bar \k(n-1,r)} =\frac1{\bar\k(n,r)} - \frac{h(\r)}{2n^2\r(1-\r)\log(2)}.
\end{align}
Therefore \eqref{knr10} proves that $\k(n,r)\geq\bar\k(n,r)$. This ends the proof, since the bound is trivially satisfied at $n=2$.

\begin{remark}\label{rem:multislice}
{\em We have restricted ourselves to the case of one type of indistinguishable particles only, but one could consider several types of particles, that is the case where there are $m$ colors and we have $r_i$ particles of color $i$, for each $i=1,\dots,m$, and $\vec{r}=(r_1,\dots,r_m)\in\bbN^m$  with $\sum_{i=1}^mr_i=n$. The case $m=2$ is covered by Theorem \ref{th:BL}. The general problem is often referred to as the {\em multislice} model. We refer to \cite{multi1,multi2} for recent work on the logarithmic Sobolev constant for the multislice. In particular, for all values of the vector $\vec{r}$, Salez \cite{multi2} obtained bounds on the  log-Sobolev constant that are tight up to a constant factor $4/\log(2)$. It is possible that a refinement of the argument in the proof of Theorem \ref{th:BL} would allow an exact computation of the entropy constant $\k(n, \vec{r})$ of the multislice with color profile $\vec r$. This is defined formally as in \eqref{knr}, but now $\mu$ is the uniform distribution over all $\binom{n}{\vec r}$ configurations, where $\binom{n}{\vec r}$ denotes the multinomial coefficient of $n$ and $\vec{r}$. 
 In this respect we may conjecture that the constant $\k(n, \vec{r})$ is again attained at a Dirac mass, which, after optimisiaztion over the color profile,  would give 
\begin{align}\label{conjmulti}
\k(n,\vec{r})=\min_{\vec{v}\sim \vec{r}}d(n,\vec{v}),\quad d(n,\vec{v}):=\frac12\sum_i \frac{v_i(n-v_i)\log(2)}{\log\binom{n}{\vec v}}
\end{align}
where $\vec{v}\sim \vec{r}$ means that $\vec{v}$ is coarser than $\vec r$, that is $\vec v\in\bbN^\ell$ with $\sum_{i=1}^\ell v_i=n$, $\ell\leq m$ is a vector obtained from $\vec r$ by repeatedly merging two entries
into one, that is by identifying certain colors.  Note that in the extremal case $m=n$, $r_i\equiv 1$, \eqref{conjmulti} coincides with the case $\ell=2$ of Theorem \ref{th:meanfield2}, while in the case $m=2$ it is equivalent to Theorem \ref{th:BL}.
}\end{remark}
\subsection{Network reduction: the case of the star graph}
Here we obtain bounds on $\k[\a,\cS_n]$ for a star graph, that is $c_{xy} = \ind_{x_*\in xy}$,  for some fixed $x_*\in V$, the center of the star. Equivalently, we take $\a_A=0$ for all $|A|\neq 2$ and $\a_{xy}=2\ind_{x_*\in xy}$. This illustrates a possible use of the network reduction approach and its shortcomings. 
\begin{proposition}\label{th:IPSn}
The star graph  $\a_A=0$ for all $|A|\neq 2$ and $\a_{xy}=2\ind_{x_*\in xy}$ satisfies
 \begin{align}\label{ipkn}
\frac{2\log 2}{\log n}\geq \k[\a,\cS_n] \geq \frac{(\log 2)^2}{\log n}.
\end{align} 
\end{proposition}
\begin{proof}
By definition, the constant $\k[\a,\cS_n]$ is the largest $\k\geq 0$ such that for all functions $f:\cS_n\mapsto\bbR_+$ one has
\begin{align}\label{ipokn}
\k\,\ent f\leq 2\sum_y \mu\left[\ent_{x_*y}f\right].
\end{align} 
The upper bound on $\k[\a,\cS_n]  $ follows by using the test function $f(\si)=\IND(\xi_1=z)$ for a fixed $z\neq x_*$. Indeed, as in Proposition \ref{prop:star} this gives $\k[\a,\cS_n]  \leq \k(S_n)  \leq \frac{2\log 2}{\log n}$. 
%for this function one has $\ent f=\frac1n\log n$, 
%$\mu\left[\ent_{x_*y}f\right]=0$ for all $y\neq z$ and $ \mu\left[\ent_{x_*z}f\right]=\log (2)/n$.

To prove the lower bound we first recall that by Theorem \ref{th:meanfield1} we know that if  
$\k_n=\k(K_n)=2(n-1)\log(2)/\log n$ as in \eqref{kKn}, then for any $\varphi:[n]\mapsto\bbR_+$ with $\sum_{i=1}^n\varphi_i=n$, 
\begin{align}\label{ipkn2}
\k_n \sum_{i=1}^n \varphi_i\log\varphi_i \leq 2\sum_{i,j=1}^n\psi(\varphi_i,\varphi_j).
\end{align}   
Therefore, for any $f\ge 0$ such that $\mu[f]=1$, for any $x\in[n]$ we have
\begin{align}\label{ipkn3}
\k_n\ent \left[\mu(f|\si_x)\right] \leq \frac2n
\sum_{i,j=1}^n\psi(\mu(f|\si_x=i),\mu(f|\si_x=j)).
\end{align}   
For all $i\neq j$ and any $x$ we have
\begin{align}\label{ipkn4}
\mu(f|\si_x=j) &= n \sum_{\si} \mu(\si)f(\si)\IND(\si_x=j) = \frac1{(n-1)!} \sum_{\si} \sum_y f(\si)\IND(\si_x=j)\IND(\si_y=i)\\& =
\frac1{(n-1)!} \sum_y \sum_{\si} f(\si^{x,y})\IND(\si_x=i)\IND(\si_y=j) = \frac1{(n-1)!} \sum_{\si} f(\si^{\xi_i,\xi_j})\IND(\si_x=i)\\& =  
\mu(f^{\xi_i,\xi_j}|\si_x=i).
\end{align}
By convexity,
\begin{align}\label{ipkn5}
\psi(\mu(f|\si_x=i),\mu(f|\si_x=j))& =  
\psi(\mu(f|\si_x=i),\mu(f^{\xi_i,\xi_j}|\si_x=i)) \\& \leq \mu(\psi(f,f^{\xi_i,\xi_j})|\si_x=i).
\end{align}
Thus, \eqref{ipkn3} shows that for any $x$:
\begin{align}\label{ipkn6}
\k_n\ent \left[\mu(f|\si_x)\right] &\leq \frac2n
\sum_{i,j=1}^n\mu(\psi(f,f^{\xi_i,\xi_j})|\si_x=i)\\&=
2\sum_{y}\mu(\psi(f,f^{x,y})) = 
2\sum_{y}\mu\left[\ent_{xy}f\right].
\end{align}   
We apply this with $x=x_*$. 
 We write
\begin{align}\label{ipkn77}
\ent f =\mu\left[\ent(f|\si_{x_*})\right]+
\ent \left[\mu(f|\si_{x_*})\right] ,
\end{align} 
and observe that by definition of the mean field constant $\k[\a^2,\cS_n] $ 
and the inequality \eqref{ipkn6} we have
\begin{align}\label{ipkn88}
\ent f &\leq \frac1{2\k[\a^2,\cS_{n-1}]}\sum_{y,z\neq x_*}
\mu\left[\ent_{yz}f\right]+\frac{2}{\k_n}\sum_{y}\mu\left[\ent_{x_*y}f\right].
\end{align} 
Next, for any $y,z$,  from \eqref{ineqents1} we have
\begin{align}\label{ipkn89}
2\log(2)\mu\left[\var_{yz}\sqrt f\right]\leq \mu\left[\ent_{yz}f\right]\leq 2\mu\left[\var_{yz}\sqrt f\right].
\end{align} 
From the octopus inequality for the variance \cite[Theorem 2.3]{CLR} it follows that  
\begin{align}\label{ipkn90}
\sum_{y,z\neq x_*}
\mu\left[\ent_{yz}f\right]
\leq 4(n-1)\sum_{y}
\mu\left[\var_{x_*y}\sqrt f\right]\leq \frac{2(n-1)}{\log(2)}\sum_{y}
\mu\left[\ent_{x_*y}f\right].
\end{align} 
In conclusion,
\begin{align}\label{ipkn91}
\ent f &\leq \left(\frac{(n-1)}{2\log(2)\k[\a^2,\cS_{n-1}]}+\frac{1}{\k_n}\right) 2\sum_{y}
\mu\left[\ent_{x_*y}f\right].
\end{align} 
From Theorem \ref{th:meanfield2} we known that $\k[\a^2,\cS_{n-1}]=\frac{\binom{n-1}2 \log (2)}{\log (n-1)!}$. Therefore, the constant $\k[\a,\cS_n]$ for the star satisfies 
\begin{align}\label{ipkn92}
\frac{1}{\k[\a,\cS_n]}&\leq \frac{(n-1)}{2\log(2)\k[\a^2,\cS_{n-1}]}+\frac{1}{\k_n}\\&=\frac{\log n}{\log 2} \left(\frac{\log((n-1)!)}{\log(2)(n-2)\log n}+\frac{1}{2(n-1)}\right)\leq \frac{\log n}{(\log 2)^2}.
\end{align} 
\end{proof}
The bounds  in Theorem \ref{th:IPSn} have a mismatch of a factor $2/\log(2)$. This could be improved to a factor 2 if one had the entropic octopus inequality \eqref{entoctopus}, since we could dispense with the estimate \eqref{ipkn89} in this case. We remark that using a Dirac mass at a given permutation shows that 
\begin{align}\label{ipkndirac}
\k[\a,\cS_n]\leq 2\log(2)\frac{(n-1)}{\log(n!)},
%\frac{(n-1)}{2\log(2)\k[\a^2,\cS_{n-1}]}+\frac{1}{\k_n}\\&=\frac{\log n}{\log 2} \left(\frac{\log((n-1)!)}{\log(2)(n-2)\log n}+\frac{1}{2(n-1)}\right)\leq \frac{\log n}{(\log 2)^2}.
\end{align} 
which is asymptotically equivalent to the constant $\k(S_n)\sim 2\log(2)/\log n$ in Proposition \ref{prop:star}, in contrast with the case of the complete graph where the Dirac mass at a permutation gives a constant that is asymtptotically twice as small as the single particle constant, see Remark \ref{rem:entperm}. However, a numerical calculation for $S_4$ shows that for the star graph one should not expect $\k[\a,\cS_n]$  and $\k(S_n)$ to be equal. 
%It seems that $2\log(2)/\log n$ should be the correct asymptotic behavior of $\k[\a,\cS_n]$  for the star and that in fact $\k[\a,\cS_n]$ may be actually equal to $\k(S_n)$.

\bibliographystyle{plain}
\bibliography{bibentropy}

\end{document}